\documentclass[letterpaper,fleqn]{article}

\usepackage[margin=1.5in]{geometry}
\usepackage{marginnote}

\usepackage{mathptmx}
\usepackage[T1]{fontenc}
\usepackage{euscript}

\usepackage{latexsym,epsfig,verbatim}
\usepackage{amsmath,amsthm,amssymb}
\usepackage{colonequals}
\usepackage{units}

\usepackage{url}
\usepackage{color}
\usepackage[colorlinks,citecolor=blue,linkcolor=red!80!black]{hyperref}
\usepackage[nameinlink]{cleveref}
\newcommand{\crefrangeconjunction}{--}
\usepackage{comment}

\usepackage{accents}

\usepackage{tikz}
\usetikzlibrary{matrix,arrows,decorations.pathmorphing}

\theoremstyle{plain}
\newtheorem{theorem}{Theorem}[section]
\newtheorem{lemma}[theorem]{Lemma}
\newtheorem{corollary}[theorem]{Corollary}
\newtheorem*{thm:progress}{\Cref{th: BF_contract}}
\newtheorem*{thm:main}{\Cref{thm:contraction_impies_progress}}
\newtheorem*{thm:orbits}{\Cref{th:contracting_orbit}}
\newtheorem*{lem:usually_nondegen}{\Cref{lem:usually_nondegenerate}}

\newtheorem{definition}[theorem]{Definition}

\newtheorem{proposition}[theorem]{Proposition}

\newtheorem*{claim}{Claim}

\theoremstyle{definition}
\newtheorem{remark}[theorem]{Remark}

\newcommand{\define}[1]{\textbf{#1}}

\newcommand{\R}{\mathbb{R}}
\newcommand{\Z}{\mathbb{Z}}

\DeclareMathOperator{\diam}{diam}

\newcommand{\ceil}[1]{\left\lceil#1\right\rceil}

\newcommand{\abs}[1]{\left\vert#1\right\vert}
\newcommand{\inv}{^{-1}}

\DeclareMathOperator{\Teich}{Teich} 
\newcommand{\T}{\Teich}

\newcommand{\C}{{\EuScript C}} 

\DeclareMathOperator{\Out}{Out}

\DeclareMathOperator{\rank}{rk}
\newcommand{\free}{\mathbb{F}} 
\newcommand{\factor}{{\EuScript F}} 
\newcommand{\F}{\factor} 
\renewcommand{\int}{\mathcal{I}}
\newcommand{\fc}{\factor} 
\newcommand{\pl}{{\EuScript{PL}}} 
\newcommand{\os}{{\EuScript X}} 
\newcommand{\X}{\os} 
\newcommand{\dsym}{d^{\mathrm{sym}}_\os} 
\newcommand{\fproj}{\pi_\fc} 
\newcommand{\plproj}{\pi_\pl} 
\newcommand{\bfproj}{\mathrm{Pr}} 
\newcommand{\len}{\ell}  
\newcommand{\cp}{\pi} 
\newcommand{\minlen}{m}  
\newcommand{\minpts}{\rho} 
\newcommand{\mintime}{\hat{\minpts}} 

\newcommand{\rose}{{\EuScript R}} 
\DeclareMathOperator{\vol}{vol} 
\newcommand{\sym}{{\sf M}} 
\newcommand{\lipconst}{{\sf L}} 
\renewcommand{\diam}{\mathrm{diam}}

\newcommand{\I}{\mathbf{I}}
\newcommand{\Ipl}{\I_+}
\newcommand{\Imin}{\I_-}
\newcommand{\J}{\mathbf{J}}

\begin{document}

\title{\textbf{\Large Contracting orbits in Outer space}}
\author{Spencer Dowdall and Samuel J. Taylor}
\date{\today}

\maketitle

\begin{abstract}
We show that strongly contracting geodesics in Outer space project to parameterized quasigeodesics in the free factor complex. This result provides a converse to a theorem of Bestvina--Feighn, and is used to give conditions for when a subgroup of $\Out(\free)$ has a quasi-isometric orbit map into the free factor complex. It also allows one to construct many new examples of strongly contracting geodesics in Outer space.
\end{abstract}


\section{Introduction}

A geodesic $\gamma \colon \I \to X$ in a metric space $X$ is strongly contracting if the closest point projection to $\gamma$ contracts far away metric balls in $X$ to sets of uniformly bounded diameter. Such geodesics exhibit hyperbolic-like behavior and are thus important to understanding the structure of the space.
This paper further develops the theory of strongly contracting geodesics in Outer space with the aim of understanding their behavior under the projection to the free factor complex. See \S\ref{sec: strong_contract} for precise definitions.  

Such geometric questions in Outer space are often motivated by their analogs in Teichm\"uller space. In that setting, strongly contracting geodesics play an important role in our understanding of the geometry of Teichm\"uller space and the mapping class groups. These geodesics are characterized by the following result of Minsky describing both their structure in Teichm\"uller space and their behavior in the curve complex. (The equivalence of $1.$ and $2.$ in \Cref{th: MM} is the main result of \cite{Minsky96} while the equivalence of $1.$ and $3.$ follows easily from Theorem $4.3$ of \cite{Minsky96}.)

\begin{theorem}[Minsky \cite{Minsky96}] \label{th: MM}
Let $\tau: \I \to \T(S)$ be a Teichm\"{u}ller geodesic. Then the following are equivalent:
\begin{enumerate}
\item There is an $\epsilon > 0$ such that $\tau$ is entirely contained in the $\epsilon$--thick part of $\T(S)$.
\item There is a $D > 0$ such that $\tau$ is a $D$--strongly contracting geodesic in $\T(S)$.
\item There is a $K \ge 1$ such that $\tau$ projects to a $K$--quasigeodesic in $\C(S)$, the curve complex of $S$.
\end{enumerate}
Moreover, the constants $\epsilon,D,K$ above depend only on each other and the topology of $S$.
\end{theorem}
Thus strongly contracting geodesics greatly illuminate the connection between Teichm\"uller space and the curve complex, as it is along these geodesics that the projection $\T(S)\to\C(S)$ is quasi-isometric.

Our main result is a version of \Cref{th: MM} for the Outer space $\X$ of a free group $\free$ and its projection $\fproj$ to the free factor complex $\F$ of $\free$. This projection has already proven to be highly useful beginning with Bestvina and Feighn's proof of hyperbolicity of the free factor complex \cite{BFhyp} and continuing with, for example, \cite{bestvina2012boundary,hamenstadt2013boundary,DT1,HorbezPoisson}. In fact, in the course of proving hyperbolicity of $\fc$, Bestvina and Feighn show that folding path geodesics which make definite progress in $\F$ are strongly contracting with respect to a specific projection tailored to folding paths. 
Combining with Theorem 4.1 and Lemma 4.11 from \cite{DT1}, this result of Bestvina and Feighn (\cite[Corollary 7.3]{BFhyp}) may be promoted to all geodesics:

\begin{theorem}\label{th: BF_contract}
Let $\gamma\colon \I \to \X$ be a geodesic whose projection to $\F$ is a $K$--quasigeodesic. Then there exists $D > 0$ depending only on $K$ (and the injectivity radius of the terminal endpoint of $\gamma$) such that $\gamma$ is $D$--strongly contracting in $\X$.
\end{theorem}

Here we prove a converse to \Cref{th: BF_contract}. Together, these establish an analog of \Cref{th: MM} in the free group setting.

\begin{theorem}\label{thm:contraction_impies_progress}
For each $D> 0$ there exist constants $K\ge 1$ and $\epsilon> 0$ with the following property. If $\gamma\colon \I\to \os$ is a nondegenerate $D$--strongly contracting geodesic, then $\gamma(\I)$ lies in the $\epsilon$--thick part $\os_\epsilon$ and $\fproj\circ\gamma\colon \I\to \fc$ is a $K$--quasigeodesic.
\end{theorem}

Recall that the Lipschitz metric on Outer space is not symmetric. Hence, a geodesic $\gamma: \I \to\X$ is not necessarily a (quasi) geodesic when traversed in the reverse direction. The condition that $\gamma$ be \define{nondegenerate} in \Cref{thm:contraction_impies_progress} is, informally, that the backwards distance along $\gamma$ meets a certain threshold depending only on $D$. See \S\ref{sec: strong_contract} for a precise definition and discussion.

\begin{remark}[Parameterized \emph{vs} Unparameterized]
In saying the projection of a geodesic to $\F$ is a $K$--quasigeodesic, we always mean a \emph{parameterized} quasigeodesic.
For any directed geodesic $\gamma\colon \I\to\os$, it is known that $\fproj\circ\gamma\colon \I\to \F$ is an \emph{unparameterized} quasigeodesic in the sense that it may be reparameterized to yield a uniform quasigeodesic \cite[Proposition 9.2]{BFhyp}. The same holds for the projections of Teichm\"uller geodesics to the curve complex \cite[Theorem 2.3]{MasurMinsky}. 
\end{remark}

\begin{remark}[Thick geodesics in $\X$]
Combining \Cref{thm:contraction_impies_progress} with \Cref{th: BF_contract} gives the three implications $(3)\iff (2) \implies (1)$ of \Cref{th: MM} in the $\Out(\free)$ setting. We stress that the implication $(1)\implies (2)$ of \Cref{th: MM} is in fact false in the Outer space setting. 
Indeed, it is well-known that there are thick geodesics in $\os$ that nevertheless project to a bounded diameter subset of $\fc$. By \Cref{thm:contraction_impies_progress} such geodesics cannot be strongly contracting. 
\end{remark}

\begin{remark}[Hyperbolic isometries of $\F$]
Combining \Cref{thm:contraction_impies_progress} with Algom-Kfir's result \cite{AKaxis} that axes of fully irreducible automorphisms in $\X$ are strongly contracting gives an alternative proof of the fact that fully irreducible automorphisms act as loxodromic isometries on $\F$ (i.e. they have positive translation length). This result was proven by Bestvina and Feighn in \cite{BFhyp} using results in \cite{bestvina2010hyperbolic}.
\end{remark}

As an application of \Cref{thm:contraction_impies_progress}, we give conditions for when the orbit map from a finitely generated subgroup $\Gamma\le\Out(\free)$ into $\F$ is a quasi-isometric embedding. First, say that $\Gamma \le \Out(\free)$ is \define{contracting in $\X$} if there exists $G \in \X$ and $D >0$ so that for any two points in the orbit $\Gamma \cdot G$ there is a $D$--strongly contracting geodesic joining them.

\begin{theorem}\label{th:contracting_orbit}
Suppose that $\Gamma \le \Out(\free)$ is finitely generated and that the orbit map $\Gamma \to \X$ is a quasi-isometric embedding. Then $\Gamma$ is contracting in $\os$ if and only if the orbit map $\Gamma\to\fc$ to the free factor complex is a quasi-isometric embedding.
\end{theorem}

We note that the ``if'' direction of \Cref{th:contracting_orbit} appeared first in our earlier work \cite{DT1} as a crucial ingredient in the proof of the following result about hyperbolic extensions of free groups:

\begin{theorem}[{\cite[Theorem 1.1]{DT1}}] \label{th:hyp_extns}
If $\Gamma\le\Out(\free)$ is purely atoroidal and the orbit map $\Gamma\to\fc$ is a quasi-isometric embedding, then the corresponding $\free$--extension $E_\Gamma$ is hyperbolic.
\end{theorem}

While the exact converse to \Cref{th:hyp_extns} is false (see \cite[\S1]{DT1}), it would nevertheless be interesting to obtain an partial converse, that is, to naturally characterize the hyperbolic extensions of $\free$ that arise from subgroups of $\Out(\free)$ that quasi-isometrically embed into $\fc$. 
It is our hope that the equivalence provided by \Cref{th:contracting_orbit} will be a useful step towards establishing such a converse.

During the completion of this paper, Hamenst\"adt and Hensel proved a result \cite[Theorem 1]{HH} that is related to \Cref{th:contracting_orbit} above.
Their theorem pertains to Morse geodesics in $\os$ and relies on Hamenst\"adt's notion of lines of minima in $\X$, introduced in \cite{Hlines}. 
We remark, however, that there is no a priori connection between strongly contracting and Morse geodesics in the asymmetric metric space $\os$ without additional assumptions on the geodesic.

\begin{remark}
\Cref{th:contracting_orbit} can be used to give new examples of strongly contracting geodesics in $\X$, beyond those which are axes of fully irreducible elements of $\Out(\free)$. Such axes were shown to be contracting by Algom-Kfir \cite{AKaxis}. For the construction, begin with a finitely generated subgroup $\Gamma \le \Out(\free)$ for which the orbit map $\Gamma \to \F$ is a quasi-isometric embedding. Many examples of such subgroups are constructed in Section $9$ of \cite{DT1}. For $R \in \X$ fixed, \Cref{th:contracting_orbit} implies that there is a $D >0$ such that for any $g,h \in \Gamma$, any directed geodesic $[g\cdot R,h\cdot R]$ is $D$--strongly contracting. Using the Arzela--Ascoli theorem, as in the proof of \cite[Proposition 5.6]{DKTct}, one may additionally take limits of such geodesics to extract bi-infinite geodesics that are $D$--strongly contracting. The geodesics constructed in this manner typically will not be axes for any fully irreducible automorphisms of $\free$.
\end{remark}

\paragraph{Outline of proof.} Let us briefly describe our approach to \Cref{thm:contraction_impies_progress}. Bestvina and Feighn's \cite{BFhyp} proof that $\fc$ is hyperbolic relies on constructing for every folding path $\gamma\colon \I\to \os$ a corresponding projection $\bfproj_\gamma\colon\fc\to\gamma(\I)$. The projection $\bfproj_\gamma$ is defined in terms of the illegal turn structure on the path $\gamma(t)$, and a careful analysis of $\bfproj_\gamma$  allows one to prove 
(i) that $\fproj\circ\bfproj_\gamma$ is a coarsely contracting retraction onto $\fproj(\gamma(\I))$ \cite[Proposition 7.2]{BFhyp} and 
(ii) that $\bfproj_\gamma\circ\fproj\colon \os\to\gamma(\I)$ coarsely agrees with the closest-point projection provided $\gamma$ makes definite progress in $\fc$ (see \cite[Lemma 4.11]{DT1}). This leads to Bestvina and Feighn's result \cite[Corollary 7.3]{BFhyp} that folding paths which make definite progress in $\fc$ are strongly contracting in $\os$ (c.f. \Cref{th: BF_contract}).

In a similar spirit, our proof of \Cref{thm:contraction_impies_progress} proceeds by constructing an appropriate projection $\minpts_\gamma\colon\pl\to\gamma(\I)$ for each geodesic $\gamma\colon \I\to \os$, where here $\pl$ is the complex of primitive conjugacy classes in $\free$ (note that $\pl$ is $\Out(\free)$--equivariantly quasi-isometric to $\fc$; see \S\ref{sec:background}). The map $\minpts_\gamma\colon \pl\to \gamma(\I)$ has a very natural definition: simply send a conjugacy class $\alpha$ to the set of points along $\gamma(\I)$ where the length of $\alpha$ is minimized. Our key technical results then show that for every $D$--strongly contracting geodesic $\gamma$, (i) the composition $\minpts_\gamma\circ \plproj\colon \os\to \gamma(\I)$ coarsely agrees with closest-point projection for distant points (\Cref{lem:projections_agree}) and  (ii) $\minpts_\gamma$ is uniformly coarsely Lipschitz (\Cref{lem:coarse_lip_projection}).

A fundamental technical difficulty arises from the fact that the Lipschitz distance $d_\os$ is highly asymmetric for points in the thin part of $\os$ (see \S\ref{sec:background}). For example: since the standard Morse lemma breaks down in the presence of boundless asymmetry, it is unclear whether strongly contracting geodesics are necessarily stable in $\os$ (that is, uniform quasigeodesics a-priori need not fellow travel strongly contracting geodesics with the same endpoints). To rule out such pathological behavior, much of the work in our discussion is devoted to proving that all nondegenerate $D$--strongly contracting geodesics lie in some uniform thick part of $\os$ (\Cref{cor:nondegen_implies_min} and \Cref{prop:left_thick}). With this key tool in hand, we are able to show (\Cref{lem:thick strong contract}) that the composition $\minpts_\gamma\circ \plproj\colon \os\to\gamma(\I)$ is a coarse retraction, meaning that points on $\gamma(\I)$ move a uniformly bounded amount. 
Combining this with the coarse Lipschitz property for $\minpts_\gamma$ (\Cref{lem:coarse_lip_projection}) then easily leads to our main result that strongly contracting geodesics in $\os$ make definite progress in $\fc$ (see \Cref{prop:strong_contract_makes_progress}).

\paragraph{Acknowledgments.} The authors thank the referee for several helpful comments. The first named author was partially supported by NSF grants DMS 1204814 and 1711089. The second named author was partially supported by NSF grants DMS 1400498 and 1744551. We also acknowledge support from NSF grants DMS 1107452, 1107263, 1107367 ``RNMS: GEometric structures And Representation varieties'' (the GEAR Network) and from the GATSBY seminar at Brown \& Yale.

\section{Background}\label{sec:background}
We briefly recall the necessary background material on the metric structure of Outer space; see \cite{FMout, BFhyp, DT1} for additional details.

\paragraph{Outer space.} Let $\free$ denote the free group of rank $r = \rank(\free) \ge 3$. 
Let $\rose$ denote the $r$--petal rose with vertex $v\in \rose$, and fix an isomorphism $\free\cong \pi_1(\rose,v)$. For our purposes, a \define{graph} is a $1$--dimensional CW complex, and a connected, simply connected graph is a \define{tree}. A \define{core graph} is a finite graph all of whose vertices have degree at least $2$.

We now define Culler and Vogtmann's \cite{CVouter} Outer space $\os$ of marked metric graphs. A \define{marked graph} $(G, g)$ is a core graph $G$ together with a \define{marking} $g\colon \rose \to G$, i.e. a homotopy equivalence. A \define{metric} on $G$ is a function $\ell\colon E(G) \to \R_{> 0}$ from the set of edges of $G$ to the positive real numbers, which assigns a length $\ell(e)$ to each edge $e \in E(G)$. The sum $\sum_{e\in E(G)} \ell(e)$ is called the \define{volume} of $G$.
With this setup, a \define{marked metric graph} is defined to be the triple $(G,g,\ell)$;  two triples $(G_1,g_1,\ell_1)$ and $(G_2,g_2,\ell_2)$ are \define{equivalent} if there is a graph isometry $\phi\colon G_1 \to G_2$ that preserves the markings in the sense that $\phi \circ g_1$ is homotopic to $g_2$. \define{Outer space}, denoted $\X$, is the set of equivalence classes of marked metric graphs of volume $1$.

The marking $\rose\to G$ for $G\in \os$ allows us to view any nontrivial conjugacy class $\alpha$ in $\free$ as a homotopy class of loops in the core graph $G$. The unique immersed loop in this homotopy class is denoted by $\alpha\vert G$, which we view as an immersion of $S^1$ into $G$. 
The \define{length of $\alpha$ in $G\in \os$}, denoted $\ell(\alpha\vert G)$, is the sum of the lengths of the edges of $G$ crossed by $\alpha\vert G$, counted with multiplicities. The \define{standard topology on $\os$} is the coarsest topology such that all of the length functions $\ell(\alpha\vert\;\cdot\;)\colon \os \to \R_+$ are continuous \cite{CVouter}. This topology agrees with other naturally defined topologies on $\X$, including the one induced by the Lipschitz metric defined below. See \cite{CVouter, paulin1989gromov, FMout} for details.

A \define{difference of markings} from $G\in \os$ to $H\in\os$ is any map $\phi\colon G\to H$ that is homotopic to $h\circ g\inv$, where $g$ and $h$ are the markings on $G$ and $H$, respectively. The \define{Lipschitz distance} from $G$ to $H$ is then defined to be
\[d_{\X}(G,H) \colonequals \inf \{ \log\left(\mathrm{Lip}(\phi)\right) \mid \phi \simeq h \circ g^{-1}\},\]
where $\mathrm{Lip}(\phi)$ denotes the Lipschitz constant of the difference of markings $\phi$. We note that while $d_\os$ is in general asymmetric (that is, $d_\os(G,H)\neq d_\os(H,G)$), it satisfies definiteness ($d_\os(G,H) = 0$ if and only if $G=H$) and the ordered triangle inequality ($d_\os(E,H) \leq d_\os(E,G) + d_\os(G,H)$) \cite{FMout}. We also have the following important result, originally due to Tad White, relating the Lipschitz distance to the ratio of lengths of conjugacy classes in the two graphs:

\begin{proposition}[See Francaviglia--Martino \cite{FMout} or Algom-Kfir \cite{AKaxis}]\label{pro: distance}
For every $G\in \os$ there exists a finite set $\mathcal{C}_G$ of primitive conjugacy classes, called \define{candidates}, whose immersed representatives in $G$ cross each edge at most twice and such that for any $H\in \os$
\[
d_\os(G,H) = \max_{\alpha\in \mathcal{C}_G} \log\left(\frac{\ell(\alpha\vert H)}{ \ell(\alpha\vert G)}\right) = \sup_{\alpha\in \free} \log\left(\frac{\ell(\alpha\vert H)}{ \ell(\alpha\vert G)}\right).\]
\end{proposition}
\noindent Note that because each candidate $\alpha \in \mathcal{C}_G$ crosses each edge of $G$ no more than twice, $\ell(\alpha|G) \le 2$.\\

Finally, a \define{geodesic in $\os$} is by definition a \emph{directed} geodesic, that is, a path $\gamma\colon \I\to \os$ such that $d_\os(\gamma(s),\gamma(t)) = t-s$ for all $s<t$. Throughout, $\I$ will always denote a closed interval $\I \subset \mathbb{R}$, and we write $\I_{\pm} \in \mathbb{R} \cup \{\pm \infty\}$ for the (possibly infinite) endpoints of the interval $\I$.

\paragraph{Asymmetry and the thick part of Outer space.} For $\epsilon > 0$, we define the \define{$\epsilon$--thick part of $\os$} to be the subset
\[\os_\epsilon \colonequals \{G\in \os : \ell(\alpha\vert G) \ge \epsilon\text{ for every nontrivial conjugacy class $\alpha$ in $\free$}\}.\]
It is also sometimes convenient to consider the symmetrization of the Lipschitz metric:
\[\dsym(G,H) \colonequals d_\os(G,H) + d_\os(H,G)\]
which is an actual metric on $\os$ and induces the standard topology \cite{FMout}. Because the Lipschitz metric $d_\os$ is not symmetric, care must be taken when discussing distances in $\os$. This asymmetry, however, is somewhat controlled in the thick part $\os_\epsilon$:

\begin{lemma}[Handel--Mosher \cite{handel2007expansion}, Algom-Kfir--Bestvina \cite{AlBest}] \label{lem: symmetric_in_thick}
For any $\epsilon>0$, there exists $\sym_\epsilon\geq 1$ so that for all $G,H\in \os_\epsilon$ we have
\[ d_\os(H,G) \le \dsym(H,G) = \dsym(G,H) \le \sym_\epsilon \cdot d_\os(G,H).\]
\end{lemma}

\noindent For $G,H \in \X$, we also use the notation $\diam_\X (G,H)$ to denote $\diam_{\X} \{G,H\} = \max \{ d_\X(G,H) ,d_\X(H,G)\} .$ Observe that $\diam_\X(G,H) \le \dsym(G,H) \le 2\diam_\X(G,H)$.

\paragraph{The factor complex.}
The main purpose of this paper is to show that strongly contracting geodesics in $\X$ project to parameterized quasigeodesics in the free factor complex, which is defined as follows: 
The \define{factor complex} $\F$ associated to the free group $\free$ is the simplicial complex whose vertices are conjugacy classes of proper, nontrivial free factors of $\free$. Vertices $[A_0] ,\ldots, [A_k]$ span a $k$--simplex if, after reordering, we have the proper inclusions $A_0 <  \dotsb < A_k$. This simplicial complex was first introduced by Hatcher and Vogtmann in \cite{HVff}. Since we are interested in the coarse geometry of $\F$, we will only consider its $1$--skeleton equipped with the path metric induced by giving each edge length $1$. The following theorem of Bestvina and Feighn is fundamental to the geometric study of $\Out(\free)$:

\begin{theorem}[Bestvina--Feighn \cite{BFhyp}]
The factor complex $\F$ is Gromov hyperbolic.
\end{theorem}

\paragraph{The primitive loop complex.}
For our proof of \Cref{thm:contraction_impies_progress}, it is more natural to work with a different complex that is nevertheless quasi-isometric to $\F$. Recall that an element $\alpha\in \free$ is \define{primitive} if it is part of some free basis of $\free$. Thus $\alpha$ is primitive if and only if $\alpha$ generates a cyclic free factor of $\free$. We use the terminology \define{primitive loop} to mean a conjugacy class of $\free$ consisting of primitive elements. The \define{primitive loop complex} $\pl$ is then defined to be the simplicial graph whose vertices are primitive loops and where two vertices are joined by an edge in $\pl$ if their respective conjugacy classes have representatives that are jointly part of a free basis of $\free$. It is straightforward to show that the natural inclusion map $\pl^0 \to \fc^0$ (each primitive conjugacy class is itself a free factor) is $2$--biLipschitz. Since the image is $1$--dense, this map is in fact a $2$--quasi-isometry. 

Relating Outer space to the primitive loop graph, we define the \define{projection} $\plproj: \X \to \pl$ in the following way: For $G \in \X$, set
\[\plproj(G) \colonequals \{\alpha \in \pl : \ell(\alpha|G) \le 2   \}. \]
This is, of course, closely related to the projection $\fproj\colon \X \to \F$ defined by Bestvina and Feighn in \cite{BFhyp} sending $G \in \X$ to the collection of free factors corresponding to proper core subgraphs of $G$. They prove that $\diam_\fc(\fproj(G))\le 4$ \cite[Lemma 3.1]{BFhyp} and that $ \diam_\fc(\alpha \cup \fproj(G))\le 6\len(\alpha\vert G) + 13$ \cite[Lemma 3.3]{BFhyp} for every $G\in \os$ and every primitive conjugacy class $\alpha$.
These estimates imply that $\plproj$ and $\fproj$ coarsely agree under the $2$--quasi-isometry $\pl\to\fc$ defined above. Combining with the fact that $\fproj$ is coarsely Lipschitz 
\cite[Corollary 3.5]{BFhyp}
this moreover gives the existence of a constant $\lipconst \ge 1$ such that 
\[d_\pl(G,H) \colonequals \diam_\pl(\plproj(G)\cup\plproj(H)) \le \lipconst\, d_\os(G,H) + \lipconst\]
for all $G,H\in \os$. That is, the projection $\plproj\colon \X \to \pl$ is coarsely $\lipconst$--Lipschitz. It is easily computed that $\lipconst \le 260$, but we prefer to work with the symbol $\lipconst$ for clarity.

\section{Strongly contracting geodesics} \label{sec: strong_contract}

Suppose that $\gamma\colon \I\to \os$ is a (directed) geodesic. Then for any point $H\in \os$ we write $d_\os(H,\gamma) = \inf\{d_\os(H,\gamma(t)) \mid t\in \I \}$ for the infimal distance from $H$ to $\gamma$. The \define{closest point projection of $H$ to $\gamma$} is then defined to be the set
\[ \cp_\gamma(H) \colonequals \{ \gamma(t) \mid t\in \I \text{ such that } d_\os(H,\gamma(t)) = d_\os(H,\gamma)\} \subset \os.\]
Note that $\pi_\gamma(H)$ could in principle have infinite diameter: due to the asymmetry of $d_\os$, the directed triangle inequality does not in general allow one to bound $d_\os(\gamma(s),\gamma(t))$ for times $s < t$ with $\gamma(s),\gamma(t)\in \pi_\gamma(H)$. On the other extreme, it is conceivable that the above infimum need not be realized (since $d_\os(H,\gamma(s))$ could remain bounded as $s\to -\infty$), in which case $\cp_\gamma(H)=\emptyset$ by definition. 

\begin{definition}[Strongly contracting]
\label{def:strongly-contracting}
A geodesic $\gamma\colon \I\to \os$ is \define{$D$--strongly contracting} if for all points $H,H'\in\os$ satisfying  $d_\os(H,H') \le d_\os(H,\gamma)$ we have
\begin{itemize}
\item $\displaystyle \diam_{\os} \left(\cp_\gamma(H) \cup \cp_\gamma(H')\right) \le D$, and 
\item $\cp_\gamma(H) = \emptyset$ if and only if $\cp_\gamma(H') = \emptyset$.
\end{itemize}
We say that $\gamma$ is strongly contracting if it is $D$--strongly contracting for some $D >  0$.
\end{definition}

We remark that the second condition is a natural extension of the first:  $\cp_\gamma(H) = \emptyset$ only if there is a sequence $s_i\in \I$ tending to $\Imin=-\infty$ with $d_\os(H,\gamma(s_i))$ limiting to $d_\os(H,\gamma)$. In this case, one should morally view $\cp_\gamma(H)$ as being ``$\gamma(-\infty)$''; hence $\diam_\os(\cp_\gamma(H)\cup\cp_\gamma(H'))$ is considered to be infinite unless $\cp_\gamma(H') = \emptyset$ as well.

While $\cp_\gamma(H)$ may in principle be empty, our first lemma shows that the closest point projection $\cp_\gamma(H)$ always exists when $\gamma$ is strongly contracting:

\begin{lemma}\label{lem:projection_exists}
If $\gamma\colon \I\to \os$ is $D$--strongly contracting, then $\cp_\gamma(H)$ is nonempty for all $H\in\os$.
\end{lemma}
\begin{proof}
Let us first show that $\cp_\gamma(H)\neq\emptyset$ for all $H$ in an open neighborhood of $\gamma(\I)$. 
Let $t\in \I$ be arbitrary and let $U\subset \I$ be an open neighborhood of $t$ whose closure $\overline{U}$ is a compact, proper subinterval of $\I$. Then there exists $C > 0$ such that $\abs{s-t} \ge C$ and consequently $\dsym(\gamma(t),\gamma(s))\ge C$ for all $s\in \I\setminus U$. In particular, we have $\gamma(s)\ne \gamma(t)$ and thus $d_\os(\gamma(t),\gamma(s)) > 0$ for all $s\in \I\setminus U$. Moreover, it is easily shown that the infimum
\[\epsilon_t = \inf\big\{d_\os(\gamma(t),\gamma(s)) : s\in \I\setminus U\big\}\]
is in fact positive. 

Now consider any point $H$ in the open neighborhood $V_t = \{y\in \os : \dsym(y,\gamma(t)) < \epsilon_t/3\}$ of $\gamma(t)$. If $d_\os(H,\gamma(s))\le\epsilon_t/2$ for some $s\in \I\setminus U$, the triangle inequality would give
\[d_\os(\gamma(t),\gamma(s)) \le \dsym(\gamma(t),H) + d_\os(H,\gamma(s)) < \epsilon_t/3 + \epsilon_t/2 <  \epsilon_t,\]
which is impossible by definition of $\epsilon_t$. Hence $d_\os(H,\gamma(s))\ge \epsilon_t/2$ for all $s\in \I\setminus U$. Since $d_\os(H,\gamma)\le d_\os(H,\gamma(t)) < \epsilon_t/3$, it follows that 
\[d_\os(H,\gamma) = \inf\{i\in \I : d_\os(H,\gamma(i))\} = \inf\{s\in \overline{U} : d_\os(H,\gamma(s))\}.\]
The above infimum is necessarily realized by compactness; thus we conclude $d_\os(H,\gamma) = d_\os(H,\gamma(s))$ for some $s\in \overline{U}$. This proves that $\cp_\gamma(H)\neq\emptyset$ for all points $H$ in the open neighborhood $V = \cup_{t\in\I} V_t$ of $\gamma(\I)$. Note that we have not yet used the assumption that $\gamma$ is strongly contracting.

Let us now employ strong contraction to complete the proof of the lemma. Let $H\in \os$ be arbitrary; we may assume $H\notin \gamma(\I)$ for otherwise the claim is obvious. Choose any path $\mu\colon [a,b]\to\os$ from $\mu(a)=H$ to some point $\mu(b)\in \gamma(\I)$. By restricting to a smaller interval if necessary, we may additionally assume that $\mu(s)\notin \gamma(\I)$ for all $a \le s < b$. Since $\cp_\gamma(G)$ is nonempty for all $G$ in a neighborhood of $\gamma(\I)$, there exists some $c \in (a,b)$ such that $\cp_\gamma(\mu(c))$ is nonempty.

Now, for each $t\in [a,c]$ we have $\mu(t)\notin\gamma(\I)$. In fact we claim that $d_\os(\mu(t),\gamma) > 0$: otherwise, as above, there would exist a sequence $s_i\in \I$ with $d_\os(\mu(t),\gamma(s_i)) \to 0$ and consequently $\gamma(s_i)\to \mu(t)$, contradicting the fact that $\gamma(\I)$ is closed. Thus for each $t\in [a,c]$ we may find an open neighborhood $W_t\subset \os$ of $\mu(t)$ such that $d_\os(\mu(t),G')\le d_\os(\mu(t),\gamma)$ for all $G'\in W_t$. 
By compactness, there is a finite subcollection $W_1,\dotsc,W_k$ of these open sets that cover $\mu([a,c])$. 
For each $i$ the strongly contracting condition now implies that  either $\cp_\gamma(G')= \emptyset$ for all $G'\in W_i$, or else $\cp_\gamma(G')\ne \emptyset$ for all $G'\in W_i$. Since these $W_1,\dotsc,W_k$ cover the connected set $\mu([a,c])$ and at least one $W_i$ falls into the latter category (namely, the set $W_i$ containing $\mu(c)$), the contingency ``$\cp_\gamma(G')\ne\emptyset$ for all $G'\in W_i$'' must in fact hold for every $i$. In particular, we see that $\cp_\gamma(H) = \cp_\gamma(\mu(a))$ is nonempty, as claimed.
\end{proof}

We will also need the following basic observation showing that the projection of a connected set to a $D$--strongly contracting geodesic $\gamma$ is effectively ``$D$--connected'' in $\gamma(\I)$:

\begin{lemma}\label{lem:dense_projection}
Let $\gamma\colon \I\to \os$ be $D$--strongly contracting and let $[a,b]\subset \I$ be any subinterval with $\abs{a-b} \ge D$. If $A\subset \os$ is connected and $\cp_\gamma(A)$ misses $\gamma([a,b])$ (that is $\cp_\gamma(A)\cap \gamma([a,b]) = \emptyset)$, then $\cp_\gamma(A)$ is either entirely contained in $\gamma(\I\cap (-\infty,a))$ or entirely contained in $\gamma(\I\cap(b,\infty))$.

\end{lemma}
\begin{proof}
Let us first establish the following
\begin{claim}
Each $H\in \os$ admits a neighborhood $U\subset \os$ with $\diam_\os(\cp_\gamma(H)\cup\cp_\gamma(H'))\le D$ for all $H'\in U$.
\end{claim}
To prove the claim, first suppose $H\notin\gamma(\I)$ so that, as above, we have $\delta\colonequals d_\os(H,\gamma) > 0$ (for otherwise there is a sequence in $\gamma(\I)$ converging to $H$). Taking the open neighborhood to be $U = \{y\in \os : d_\os(H,y)<\delta\}$, the strongly contracting condition then ensures $\diam_\os(\cp_\gamma(H),\cp_\gamma(H'))\le D$ for all $H'\in U$. Next suppose $H\in \gamma(\I)$ so that $d_\os(H,\gamma) = 0$ and $\cp_\gamma(H) = \{H\}$. Choosing $\epsilon>0$ so that $H\in \os_{2\epsilon}$, we may then choose $\delta > 0$ sufficiently small so that $\delta\sym_\epsilon < D/2$ and the entire neighborhood $U = \{y\in \os : \dsym(y,H)< \delta\}$ is contained in $\os_{\epsilon}$ (where $\sym_\epsilon$ is the symmetrization constant from \Cref{lem: symmetric_in_thick}). For any $H'\in U$ we then have $d_\os(H',\gamma)\le d_\os(H',H)  < \delta$. Thus any $G\in \cp_\gamma(H')$ satisfies $\dsym(G,H')\le \sym_\epsilon d_\os(H',G) < \sym_\epsilon \delta < D/2$. Note that we also have $\dsym(H,H') < \delta < D/2$. Since $\cp_\gamma(H) = \{H\}$, the triangle inequality therefore shows the desired inequality $\diam_\os(\cp_\gamma(H)\cup\cp_\gamma(H'))\le D$. Since this holds for each $H'\in U$, the claim follows.

We now prove the lemma. Let $[a,b]\subset \I$ and $A\subset \os$ be as in the statement of the lemma, so that $\cp_\gamma(A)$ is disjoint from $\gamma([a,b])$. Since $\cp_\gamma(H)$ is always nonempty (by \Cref{lem:projection_exists}) and satisfies $\diam_\os(\cp_\gamma(H))\le D$, 
the hypothesis on $\cp_\gamma(A)\cap \gamma([a,b])$ implies that each $H\in A$ lies in exactly one of two the sets
\[A_-=\{H\in A : \cp_\gamma(H)\subset \gamma(\I\cap(-\infty,a))\}\qquad\text{or}\qquad A_+=\{H\in A : \cp_\gamma(H)\subset \gamma(\I\cap(b,\infty))\}.\]
Thus $A = A_-\cup A_+$ gives a partition of $A$. Moreover, the claim proves that $A_-$ and $A_+$ are both open. The connectedness of $A$ therefore implies that either $A_-$ or $A_+$ is empty, which is exactly the conclusion of the lemma.
\end{proof}

Finally, we say that a $D$--strongly contracting geodesic $\gamma\colon \I\to \os$ is \define{nondegenerate} if there exists times $s < t$ in $\I$ such that $d_\os(\gamma(t),\gamma(s)) \ge 18D\lipconst$. \Cref{lem:usually_nondegenerate} below shows that this mild symmetry condition automatically holds in most natural situations.

\section{Length minimizers}
\label{sec:length_min}
To control the nearest point projection of a graph $H$ to a geodesic $\gamma$, we must understand where the lengths of conjugacy classes in $\free$ are minimized along $\gamma$. To this end, we introduce the following terminology:  Firstly, given a directed geodesic $\gamma\colon \I\to \os$ and a nontrivial conjugacy class $\alpha\in \free$, we typically write
\[ \minlen_\alpha = \minlen_{\alpha}^{\gamma} \colonequals \inf_{t\in \I} \ell(\alpha\vert \gamma(t))\]
for the infimal length that the conjugacy class attains along $\gamma$. We then regard the set 
\[\minpts_\gamma(\alpha) = \{x\in \gamma(\I) \mid \ell(\alpha\vert x) = \minlen_\alpha\}\]
as the \define{projection of $\alpha$ to $\gamma$}. Since it is possible to have $\minpts_\gamma(\alpha)=\emptyset$ in the case that $\I$ is not compact, we also define a parameterwise-projection 
\[\mintime'_\gamma(\alpha) = \big\{t \mid \exists s_i\in \I \text{ s.t. } s_i\to t \text{ and } \ell(\alpha\vert\gamma(s_i)) \to \minlen_\alpha\} \subset [-\infty,+\infty].\]
For technical reasons, it is convenient to instead work with the following variant:
\[\mintime_\gamma(\alpha) \colonequals \begin{cases}
\mintime'_\gamma(\alpha)\cap \R, & \minpts_\gamma(\alpha)\neq \emptyset\\
\mintime'_\gamma(\alpha),& \text{else}.
\end{cases}\]
Thus $\mintime_\gamma(\alpha)$ is never empty and is exactly the set of parameters $t\in \I$ realizing $\minlen_\alpha$ when $\minpts_\gamma(\alpha)$ is nonempty. 
Note also that $\gamma(\mintime_\gamma(\alpha)\cap \R) = \minpts_\gamma(\alpha)$ in all cases.

As indicated above, we think of $\minpts_\gamma$ as a projection from the set of conjugacy classes in $\free$ onto $\gamma$. The next lemma shows that for strongly contracting geodesics, $\minpts_\gamma$ is compatible with closest-point projection $\cp_\gamma$ in the sense that graphs $H\in \os$ and embedded loops in $H$ often coarsely project to the same spot. First, notice that every metric graph in $\os$ contains an embedded loop of length at most $\nicefrac{2}{3}$, and that this loop necessarily defines a primitive conjugacy class of $\free$.

\begin{lemma}[Projections agree]
\label{lem:projections_agree}
Let $\gamma\colon \I\to \os$ be a $D$--strongly contracting geodesic. Suppose that $H\in \os$ is such that $d_\os(H,\gamma)\ge \log(3)$. Then for every conjugacy class $\alpha$ corresponding to an embedded loop in $H$ with $\len(\alpha\vert H) \le \nicefrac{2}{3}$, we have that $\minpts_\gamma(\alpha)\neq\emptyset$ and that 
\[\diam_\os(\minpts_\gamma(\alpha)\cup\cp_\gamma(H))\le D.\]
Moreover, there exists a primitive conjugacy class $\alpha$ satisfying these conditions.
\end{lemma}
\begin{proof}
\Cref{lem:projection_exists} ensures the existence of a time $t\in \I$ so that $d_\os(H,\gamma) = d_\os(H,\gamma(t))$.  Write $G= \gamma(t)$. Let $\alpha$ be a conjugacy class corresponding to an embedded loop in $H$ with $\len(\alpha\vert H)\le \nicefrac{2}{3}$. 

To prove that $\alpha$ satisfies the conclusion of the lemma, for $0 < \sigma \leq 1$, let $H_\sigma$ denote the metric graph obtained from $H$ by scaling the edges comprising $\alpha\vert H\subset H$ by $\sigma$ and scaling all other edges by $\frac{1-\sigma\len(\alpha\vert H)}{1-\len(\alpha\vert H)}$ (so as to maintain $\vol(H_\sigma) = 1$). It follows that
\[d_\os(H,H_\sigma) \le \log\left(\frac{1- \sigma \len(\alpha\vert H)}{1 - \len(\alpha\vert H)}\right) \le \log\left(\frac{1}{1 - \len(\alpha\vert H)}\right) \le \log(3).\]
Since $\gamma$ is $D$--strongly contracting and $d_\os(H,\gamma)\ge \log(3)$ by hypothesis, we may conclude that $\diam_\os(\cp_\gamma(H)\cup\cp_\gamma(H_\sigma))\le D$ for all $0 < \sigma \le 1$.

First suppose that $\minpts_\gamma(\alpha)\neq\emptyset$. Letting $B\in \minpts_\gamma(\alpha)$ be arbitrary, we then have  $\len(\alpha\vert B) = \minlen_\alpha$. Let $\mathcal{C}$ be the set of candidates of $H$; this is also the set of candidates for each $H_\sigma$. Since $\alpha$ is embedded in $H$, it is the only candidate whose edge lengths all tend to zero as $\sigma\to 0$. Thus for every candidate $z\neq \alpha\in \mathcal{C}$, there is a positive lower bound on $\len(z\vert H_\sigma)$ as $\sigma\to 0$. On the other hand, $\len(\alpha\vert H_\sigma)$ clearly tends to zero as $\sigma \to 0$. 
For $\sigma >0$ sufficiently small, \Cref{pro: distance} therefore gives
\[\log\left(\frac{\minlen_\alpha}{\len(\alpha\vert H_\sigma)}\right) = \log\left(\frac{\len(\alpha\vert B)}{\len(\alpha\vert H_\sigma)}\right) = \log\left(\max_{z\in \mathcal{C}} \frac{\len(z \vert B)}{\len(z\vert H_\sigma)}\right) = d_\os(H_\sigma,B).\]
The fact that $\minlen_\alpha$ is the minimal length achieved by $\alpha$ along $\gamma$ moreover implies that $d_\os(H_\sigma,\gamma) \ge \log(\minlen_\alpha / \ell(\alpha\vert H_\sigma))$. Therefore $d_\os(H_\sigma, B) = d_\os(H_\sigma,\gamma)$, and so we may conclude $B\in \cp_\gamma(H_\sigma)$. This shows that $\minpts_\gamma(\alpha)\subset \cp_\gamma(H_\sigma)$ and therefore that $\diam_\os(\minpts_\gamma(\alpha)\cup\cp_\gamma(H))\le D$  in the case that $\minpts_\gamma(\alpha)$ is nonempty.

It remains to rule out the possibility that $\minpts_\gamma(\alpha)$ is empty. Since in this case the infimal length $\minlen_\alpha$ is \emph{only} achieved by sequences of times tending to $\pm\infty$, we may choose $\epsilon > 0$ sufficiently small so that for any $s\in \I$ we have the implication
\[\len(\alpha\vert \gamma(s)) < \minlen_\alpha + \epsilon \implies \abs{s-t} > 2D.\]
Let us choose such a time $s_0\in\I$ with $\len(\alpha\vert \gamma(s_0)) < \minlen_\alpha + \epsilon$. As above, by taking $\sigma > 0$ sufficiently small we may be assured that 
\[d_\os(H_\sigma,\gamma(s_0)) =  \log\left(\max_{z\in \mathcal{C}} \frac{\len(z\vert \gamma(s_0))}{\len(z\vert H_\sigma)}\right) =  \log\left(\frac{\len(\alpha\vert \gamma(s_0))}{\len(\alpha\vert H_\sigma)}\right) .\]
Since $\cp_\gamma(H_\sigma)$ is nonempty by \Cref{lem:projection_exists}, there exists a time $s\in \I$ for which $\gamma(s)\in \cp_\gamma(H_\sigma)$. Since this is a \emph{closest} point on $\gamma$ from $H_\sigma$, we necessarily have $d_\os(H_\sigma,\gamma(s))\le d_\os(H_\sigma,\gamma(s_0))$ which in turn requires $\len(\alpha\vert \gamma(s)) \le \len(\alpha\vert \gamma(s_0)) < m_\alpha+\epsilon$. By the choice of $\epsilon$, this implies $\abs{s-t} > 2D$ and consequently $\diam_\os(\gamma(s),\gamma(t))> 2D$. However, since $\gamma(s)\in \cp_\gamma(H_\sigma)$ and $\gamma(t) \in \cp_\gamma(H)$, this contradicts the fact that $\diam_\os(\cp_\gamma(H)\cup\cp_\gamma(H_\sigma))\le D$ for all $0 < \sigma \le 1$. Therefore $\minpts_\gamma(\alpha)$ cannot be empty, and the lemma holds.
\end{proof}

A priori, it could be that $\minpts_\gamma(\alpha)$ is empty for \emph{every} nontrivial conjugacy class in $\free$ and, in keeping with \Cref{lem:projections_agree}, that all points of $\os$ lie within $\log(3)$ of $\gamma(\I)$. Our next lemma rules out such pathological behavior.

\begin{lemma}[Some projection exists]
\label{lem:one_minimizer_exists}
Let $\gamma\colon\I\to\os$ be a strongly contracting geodesic. Then there exists a primitive conjugacy class $\alpha\in\pl^0$ such that $\minpts_\gamma(\alpha)$ is nonempty and $\minlen_\alpha > 0$.
\end{lemma}

\begin{proof}
It suffices to find $H \in \X$ with $d_{\X}(H,\gamma) \ge \log(3)$, for then \Cref{lem:projections_agree} will provide a class $\alpha \in \pl^0$ with $\rho_\gamma(\alpha) \neq \emptyset$. However, note that if $d_\X(H,\gamma) \le \log(3)$, then $d_\pl(H, \gamma) \le \lipconst \cdot \log(3)+\lipconst$, where $\lipconst$ is the coarse Lipschitz constant for the projection $\pi_\pl\colon \X \to \pl$.

By Proposition $9.2$ of \cite{BFhyp} the projection $\pi_\pl(\gamma)$ is a uniform unparameterized quasigeodesic in $\pl$. Since $\pl$ has infinite diameter and is not quasi-isometric to $\Z$ (see, for example, Theorem $9.3$ of \cite{BFhyp}), we can use surjectivity of $\pi_\pl\colon \X\to\pl$ to choose $H \in \X$ with $d_\pl(H,\gamma) > 2\lipconst \cdot \log(3)$. Combining this fact with the observation above completes the proof of the lemma.
\end{proof}

With these basic properties of $\minpts_\gamma$ established, we now turn to the main ingredient in the proof of \Cref{thm:contraction_impies_progress}. Restricting to the primitive conjugacy classes, our construction of $\mintime_\gamma$ (or alternately $\minpts_\gamma)$ thus gives a projection $\mintime_\gamma\colon \pl\to \mathcal{P}(\I)$ for each geodesic $\gamma\colon \I\to \os$. 
Our next lemma shows that $\minpts_\gamma$ is in fact uniformly Lipschitz provided $\gamma$ is strongly contracting. 
Note that this \emph{does not} yet show that the projection $\rho_\gamma$ is a retraction.

\begin{lemma}[$\minpts_\gamma$ is coarsely Lipschitz]
\label{lem:coarse_lip_projection}
Suppose that $\gamma\colon \I\to \os$ is a $D$--strongly contracting geodesic and let $\alpha,\beta\in\pl^0$ be primitive loops. Then $\minpts_\gamma(\alpha)$ is nonempty (so $\mintime_\gamma(\alpha)\subset\R$) and 
\[\displaystyle \diam_\os\big(\minpts_\gamma(\alpha)\cup \minpts_\gamma(\beta)\big) \le D \cdot d_\pl(\alpha,\beta) + D.\]
\end{lemma}
\begin{proof}
The lemma will follow easily from the following fact:
\begin{claim}
If $\minpts_\gamma(\alpha)\neq\emptyset$ and $\beta\in \pl^0$ is adjacent to $\alpha$, then $\minpts_\gamma(\beta)\neq \emptyset$ and $\diam_\os\big(\minpts_\gamma(\alpha)\cup\minpts_\gamma(\beta)\big)\le D$.
\end{claim}

Indeed, since $\pl$ is connected and there exists $\alpha_0\in\pl^0$ with $\minpts_\gamma(\alpha_0)\neq\emptyset$ by \Cref{lem:one_minimizer_exists}, the claim shows that $\minpts_\gamma(\beta)$ is nonempty for every primitive loop $\beta\in\pl^0$. One may thus deduce the lemma by applying the claim inductively with the triangle inequality to obtain the desired bound $\diam_\os(\minpts_\gamma(\alpha)\cup\minpts_\gamma(\beta)) \le D\cdot d_\pl(\alpha,\beta) + D$ for all $\alpha,\beta\in\pl^0$.

It remains to prove the claim.
First, we may choose a free basis $\{e_1,\dotsc,e_n\}$ of $\free$ in which $e_1$ represents the conjugacy class $\alpha$ and $e_2$ represents the conjugacy class $\beta$. Let $(R,g)$ be the marked rose with petals labeled by the basis elements $e_1,\dotsc, e_n$. By \Cref{pro: distance} there is a finite set $\mathcal{C}$ of candidate conjugacy classes represented by immersed loops in $R$ such that for any metric $\ell$ on $R$ the distance to any other point $H\in \os$ is given by
\[d_\os((R,g,\ell),H) = \log\left(\sup_{z\in \mathcal C} \frac{\ell(z \vert H)}{\ell(z \vert (R,g,\ell))}\right).\]
Furthermore, both $\alpha$ and $\beta$ are candidates since they label petals of $R$. 

Choose an arbitrary point $G = \gamma(t) \in \minpts_\gamma(\alpha)$. If $\minpts_\gamma(\beta)$ is empty, then $\mintime_\gamma(\beta) \subset \{-\infty,+\infty\}$ meaning that the infimal length $\minlen_\beta$ of $\beta$ is \emph{only} achieved by sequences of times tending to $\pm\infty$; in which case we may chose $\epsilon_0 > 0$ sufficiently small so that for any $s_0\in \I$ we have the implication
\begin{equation}\label{eqn:epsilon_0_choice}
\len(\beta\vert \gamma(s_0)) < \minlen_\beta + \epsilon_0 \implies \abs{s_0-t}> 2D.
\end{equation}
We now fix a time parameter $s\in \I$ as follows: If $\minpts_\gamma(\beta)\neq\emptyset$, then we choose $s\in \mintime_\gamma(\beta)$ arbitrarily; if $\minpts_\gamma(\beta)= \emptyset$, we instead let $s\in \I$ be any time for which $\len(\beta\vert \gamma(s))< \minlen_\beta+\epsilon_0$. In either case we set $H = \gamma(s)$.

Let us now define 
\begin{align*}
K \colonequals& \max \big\{ \ell(z\vert Y) \mid z\in \mathcal{C}\text{ and }Y\in \{G,H\}\big\} \quad\text{and}\\
k \colonequals& \min \big\{ \ell(z\vert Y) \mid z\in \mathcal{C}\text{ and }Y\in \{G,H\}\big\}
\end{align*}
to be the maximal and minimal lengths achieved by any candidate $z\in \mathcal{C}$ at either $G$ or $H$. Notice that $K\ge k > 0$. Choose a small parameter  $0 <  \delta < \frac{k}{2rK} \le \frac{1}{2}$.
For each $0 < \sigma < 1$ we let $R_\sigma$ denote the marked metric graph $(R,g,\ell_\sigma)$ in which the petals of $R$ corresponding to $\alpha$ and $\beta$ have lengths $\ell(\alpha\vert R_\sigma) = \sigma\delta$ and $\ell(\beta \vert R_\sigma) = (1-\sigma)\delta$, respectively, and every other petal of $R_\sigma$ has length $\frac{1-\delta}{r-2}$ (so that $\vol(R_\sigma) = 1$). Notice that we then have $\ell(z\vert  R_\sigma) \ge \frac{1-\delta}{r-2} \ge \frac{1}{2r}$ for \emph{every} candidate $z\in \mathcal C$ except for the candidates $\alpha$ and $\beta$. 

Let us now estimate the distance from $R_\sigma$ to points along $\gamma(\I)$. 
Firstly, at $G\in \gamma(\I)$ we have $\len(\alpha\vert G) = \minlen_\alpha$ and $\len(z\vert G) \le K$ for all other candidates $z\in \mathcal{C}$. Thus we have 
\begin{align*}
d_\os(R_\sigma, G) =& \log\sup_{z\in \mathcal C} \frac{\ell(z \vert G)}{\ell(z\vert R_\sigma)}
\le \log\sup\left\{ \frac{\len(\alpha\vert G)}{\delta \sigma}, \frac{K}{(1-\sigma)\delta}, 2rK\right\}.
\end{align*}
Since $\len(\alpha\vert G) \ge k$ and $\sigma\delta < \delta < \frac{k}{2rK}$, when $\frac{\sigma}{1-\sigma} < \frac{k}{K}$ the above estimate reduces to
\[d_\os(R_\sigma,G) = \log\left(\frac{\len(\alpha\vert G)}{\sigma\delta}\right) = \log\left(\frac{\minlen_\alpha}{\sigma\delta}\right).\]
Since $\minlen_\alpha$ is the minimal length of $\alpha$ achieved on $\gamma(\I)$, we also have
\begin{equation}\label{eqn:alpha_dist}
d_\os(R_\sigma,\gamma(u)) \ge \log\left(\frac{\len(\alpha\vert \gamma(u))}{\len(\alpha\vert R_\sigma)}\right) \ge \log\left(\frac{\minlen_\alpha}{\sigma\delta}\right) = d_\os(R_\sigma,G)
\end{equation}
for all $u\in \I$. Therefore $G\in \cp_\gamma(R_\sigma)$ whenever $\frac{\sigma}{1-\sigma} < \frac{k}{K}$.  

A similar argument shows that
\begin{equation}\label{eqn:beta_dist}
d_\os(R_\sigma, H)  = \log\sup\left\{\frac{K}{\delta\sigma},\frac{\ell(\beta\vert H)}{(1-\sigma)\delta}, 2rK\right\} = \log\left(\frac{\len(\beta\vert H)}{(1-\sigma)\delta}\right)
\end{equation}
whenever $\frac{1-\sigma}{\sigma} < \frac{k}{K}$. If $\minpts_\gamma(\beta)$ is nonempty, so that $\len(\beta\vert H) = \minlen_\beta$ by choice of $H$, we again conclude $H\in \cp_\gamma(R_\sigma)$ whenever $\frac{1-\sigma}{\sigma}< \frac{k}{K}$. Otherwise, we note that any closest point $H_\sigma \in \cp_\gamma(R_\sigma)$ satisfies $d_\os(R_\sigma,H_\sigma)\le d_\os(R_\sigma,H)$ and consequently $\len(\beta\vert H_\sigma)\le \len(\beta\vert H) < \minlen_\beta + \epsilon_0$ by the choice of $H$.

Let us now specify parameters $0 < \sigma_\alpha < \sigma_\beta < 1$ by the formulas
\[\frac{\sigma_\alpha}{1-\sigma_\alpha} = \frac{k}{2K} \qquad\text{and}\qquad \frac{1-\sigma_\beta}{\sigma_\beta} = \frac{k}{2K}.\]
Setting $R_\alpha = R_{\sigma_\alpha}$ and $R_\beta = R_{\sigma_\beta}$, equation \eqref{eqn:alpha_dist} ensures that $G\in \cp_\gamma(R_\alpha)$. A comparison of lengths of candidates at $R_\alpha$ and $R_\beta$ shows that
\[d_\os(R_\alpha,R_\beta) = \log\left(\frac{\ell(\alpha\vert R_\beta)}{\ell(\alpha\vert R_\alpha)}\right) = \log\left(\frac{\sigma_\beta\delta}{\sigma_\alpha \delta}\right) = \log\left(\frac{2K}{k}\right) \]
and similarly $d_\os(R_\beta,R_\alpha) = \log\left(\frac{1-\sigma_\alpha}{1-\sigma_\beta}\right) = \log(2K/k)$. Since this is independent of $\delta$, by choosing $\delta$ sufficiently small we can ensure that
\[d_\os(R_\alpha,R_\beta) = \log(2K/k) < \log(\minlen_\alpha/\delta \sigma_\alpha) = d_\os(R_\alpha,\gamma).\]
Therefore, the $D$--strongly contracting property implies that
\begin{equation}\label{eqn:diameterbound_for_lipschitz_lemma}
\diam_\os(\cp_\gamma(R_\alpha)\cup \cp_\gamma(R_\beta)) \le D.
\end{equation}

We now finish proving the claim: First consider the case $\minpts_\gamma(\beta)\neq \emptyset$, so that $s\in \mintime_\gamma(\beta)$ by the choice of $s$ and consequently $H = \gamma(s)\in \cp_\gamma(R_\beta)$ by equation \eqref{eqn:beta_dist}. Therefore equation \eqref{eqn:diameterbound_for_lipschitz_lemma} ensures $\diam_\os(H,G)\le D$. Since $G=\gamma(t)\in \minpts_\gamma(\alpha)$ and $H = \gamma(s)\in \minpts_\gamma(\beta)$ were chosen arbitrarily, this proves the claim in the case that $\minpts_\gamma(\beta)$ is nonempty. It remains to rule out the possibility $\minpts_\gamma(\beta) = \emptyset$, in which case our choice of $H = \gamma(s)$ gives $\len(\beta\vert H) < \minlen_\beta + \epsilon_0$. Let $H' = \gamma(s')\in \cp_\gamma(R_\beta)$ be any closest point from $R_\beta$. We then have $d_\os(R_\beta,H')\le d_\os(R_\beta,H)$ which, by equation \eqref{eqn:beta_dist}, in turn implies
\[\len(\beta\vert H') \le \len(\beta \vert H) < \minlen_\beta + \epsilon_0.\]
Our choice of $\epsilon_0$ \eqref{eqn:epsilon_0_choice} then ensures that $\abs{s'- t} > 2D$. However, since $\gamma(s')\in \cp_\gamma(R_\beta)$ and $\gamma(t)\in \cp_\gamma(R_\alpha)$, this contradicts \eqref{eqn:diameterbound_for_lipschitz_lemma}. Thus the contingency $\minpts_\gamma(\beta) = \emptyset$ is impossible and the claim holds.
\end{proof}

\begin{remark}
\Cref{lem:coarse_lip_projection} and the fact that $\plproj\colon \os\to \pl$ is coarsely $\lipconst$--Lipschitz together imply, as in the proof of \Cref{prop:strong_contract_makes_progress} below, that any $D$--strongly contracting geodesic in $\os$ in fact has $D\ge 1/\lipconst$ or else has uniformly bounded diameter. However we will not use this fact going forward.
\end{remark}

\section{The progression of thick, strongly contracting geodesics}
In this section, we prove our main theorem in the case that the geodesic is contained in some definite thick part of $\X$. The arguments in this case are made easier by the fact that we can first prove that the diameter of times for which a fixed conjugacy class has bounded length is uniformly controlled. This is the content of \Cref{lem:thick strong contract}.

\begin{proposition}[Thick strongly contracting geodesics make progress in $\fc$]
\label{prop:strong_contract_makes_progress}For each $D> 0$ and $\epsilon > 0$ there exists a constant $K = K(D,\epsilon) \ge 1$ with the following property. If $\gamma\colon \I\to \os$ is a $D$--strongly contracting geodesic and $\gamma(\I)\subset\os_\epsilon$, then $\fproj\circ\gamma\colon \I\to \fc$ is a $K$--quasigeodesic.
\end{proposition}

Before proving \Cref{prop:strong_contract_makes_progress},  recall that given a directed geodesic $\gamma\colon \I\to \os$ and a nontrivial conjugacy class $\alpha\in \free$, we write $m_\alpha = \inf_{t\in \I} \ell(\alpha\vert \gamma(t))$ for the infimal length that the conjugacy class attains along $\gamma$. In the case of a thick strongly contracting geodesic, the set of times where $\alpha$ is short is controlled as follows:

\begin{lemma}[Transient shortness]\label{lem:thick strong contract}
Suppose that $\gamma\colon \I\to \os$ is a $D$--strongly contracting geodesic with $\gamma(\I)\subset \os_\epsilon$, and set $\epsilon' = \epsilon/(1+2\epsilon\inv)$. Then for every primitive element $\alpha \in \free$ we have
\[\diam_\os\Big\{\gamma(s) : s\in \I \text{ and } \ell(\alpha\vert\gamma(s)) \le m_\alpha + 2\Big\} \le 2\sym_{\epsilon}(1+\sym_{\epsilon'})\log\left(1+\tfrac{2}{\epsilon}\right) + 2D\]
where $\sym_\epsilon$ is the symmetrization constant provided by \Cref{lem: symmetric_in_thick}.
\end{lemma}
\begin{proof}
Suppose that $G,H\in \gamma(\I)$ are points for which $\ell(\alpha\vert G),\ell(\alpha\vert H) \le m_\alpha + 2$. 
Fix a free basis $A =\{e_1,\dotsc,e_r\}$ of $\free$ with $e_1 = \alpha$ and let $(R,g)$ be the marked rose with petals labeled by elements of $A$. Let $\mathcal{C} = \mathcal{C}_R$ denote the finite set of candidates of $R$. For each $0 < \sigma < \nicefrac{1}{2}$, let $R_\sigma\in \os$ denote the marked metric graph $(R,g,\ell_\sigma)$ in which the petal labeled $\alpha$ has $\ell(\alpha\vert R_\sigma) = \sigma$ and every other petal has length $(1-\sigma)/(r-1)$. Notice that $\mathcal{C}$ is the set of candidates for each metric graph $R_\sigma$, and that we moreover have $\ell(z\vert R_\sigma) \ge (1-\sigma)/(r-1)\ge \frac{1}{2r}$ for \emph{every} candidate $z\in \mathcal{C}$ except for $\alpha$. 

We henceforth suppose our parameter satisfies $\sigma < m_\alpha$. By definition of $m_\alpha$ we thus have $\ell(\alpha\vert \gamma(t))/\ell(\alpha\vert R_\sigma)\ge \frac{m_\alpha}{\sigma}$ for all $t\in \I$; hence $d_\os(R_\sigma,\gamma) \ge \log(m_\alpha/\sigma)$. Now consider the maximum length
\[M = \max_{z\in \mathcal{C}}\left\{\ell(z\vert G), \ell(z\vert H)\right\}\]
achieved by any candidate at the two points $G,H$. If $\sigma$ is additionally chosen so that $\sigma < m_\alpha/(2rM)$, then we see that for all candidates $\alpha\neq z\in \mathcal{C}$ we have
\[\frac{\ell(\alpha\vert G)}{\ell(\alpha\vert R_\sigma)}\ge \frac{m_\alpha}{\sigma} \ge \frac{M}{\nicefrac{1}{2r}} \ge \frac{\ell(z\vert G)}{\ell(z\vert R_\sigma)}
\quad\text{and}\quad
\frac{\ell(\alpha\vert H)}{\ell(\alpha\vert R_\sigma)}\ge \frac{m_\alpha}{\sigma} \ge \frac{M}{\nicefrac{1}{2r}} \ge \frac{\ell(z\vert H)}{\ell(z\vert R_\sigma)}.\]
It now follows from \Cref{pro: distance} that
\[\log\left(\frac{m_\alpha}{\sigma}\right) \le d_\os(R_\sigma,\gamma) \le 
\left\{\begin{array}{c} 
d_\os(R_\sigma,G) = \log\frac{\ell(\alpha\vert G)}{\ell(\alpha\vert R_\sigma)}\vspace{3pt}\\
d_\os(R_\sigma,H) = \log\frac{\ell(\alpha\vert H)}{\ell(\alpha\vert R_\sigma)}
\end{array}\right\}\le \log\left(\tfrac{m_\alpha + 2}{\sigma}\right).\]

Choose a directed geodesic $\rho\colon [0,K] \to \os$ from $R_\sigma$ to $G$, where $K = d_{\X}(R_\sigma,G)$. Since $e^{d_\os(R_\sigma,G)} = \frac{\ell(\alpha\vert G)}{\ell(\alpha\vert R_\sigma)}$ and $\rho$ is a geodesic for the Lipschitz metric, it follows that $\ell(\alpha\vert \rho(t)) = \sigma e^t$ for all $t\in [0,K]$. Hence if we define $G' = \rho\left(\log(m_\alpha / \sigma)\right)$ (so that $\ell(\alpha\vert G') = m_\alpha$), we see that
\begin{align} \label{eq: distance_along}
d_\os(R_\sigma,G') = \log\left(\tfrac{m_\alpha}{\sigma}\right) \le d_\os(R_\sigma,\gamma)\quad\text{and}\quad d_\os(G',G) = \log\left(\tfrac{\ell(\alpha\vert G)}{m_\alpha}\right) \le \log\left(1 + \tfrac{2}{\epsilon}\right).
\end{align}
Defining $H'$ (on the geodesic from $R_\sigma$ to $H$) similarly, we obtain analogous inequalities for $H'$.
By the strongly contracting condition, the first inequality of (\ref{eq: distance_along}) shows that $\diam_\os(\pi_\gamma(G')\cup \pi_\gamma(R_\sigma))$ and similarly $\diam_\os(\pi_\gamma(H') \cup \pi_\gamma(R_\sigma))$ are both bounded by $D$. Whence
\begin{equation}\label{eqn:close_projections}
\diam_\os\big(\pi_\gamma(G')\cup\pi_\gamma(H')\big) \le 2D.
\end{equation}
On the other hand, the second inequality of (\ref{eq: distance_along}) shows $G',H'\in \os_{\epsilon'}$, where $\epsilon' = \epsilon/(1+2\epsilon\inv)$. Therefore, we also have
\[d_\os(G,G'),\; d_\os(H,H') \le \sym_{\epsilon'}\log\left(1+\tfrac{2}{\epsilon}\right).\]
Choose any points $G_0\in \pi_\gamma(G')$ and $H_0\in \pi_\gamma(H')$. Since these are by definition \emph{closest} points, $d_\os(G',G_0)$ and $d_\os(H',H_0)$ can be at most $\log\left(1 + \tfrac{2}{\epsilon}\right)$. By the triangle inequality, it follows that
\[d_\os(G,G_0),\; d_\os(H,H_0)\le (1 + \sym_{\epsilon'})\log\left(1+\tfrac{2}{\epsilon}\right).\]
Symmetrizing (\Cref{lem: symmetric_in_thick}) to obtain bounds on $\diam_\os(G,G_0)$ and $\diam_\os(H,H_0)$ and combining with \eqref{eqn:close_projections}, another application of the triangle inequality now gives
\[\diam_\os(G,H) \le 2\sym_{\epsilon}(1+\sym_{\epsilon'})\log\left(1+\tfrac{2}{\epsilon}\right) + 2D \qedhere.\]
\end{proof}

Since each primitive loop $\alpha$ in the projection $\plproj(G_t)$ of $G_t=\gamma(t)$ satisfies $\ell(\alpha\vert G_t) \le \minlen_\alpha +2$ by definition, \Cref{lem:thick strong contract} shows that the composition
\[\gamma(\I) \overset{\plproj}{\longrightarrow}\pl\overset{\minpts_\gamma}{\longrightarrow}\gamma(\I)\]
moves points a uniformly bounded distance depending only on $D$ and $\epsilon$. That is, for each thick strongly contracting geodesic $\gamma\colon \I\to \os$, the composition $\plproj\circ\minpts_\gamma$ gives a coarse retraction from $\pl$ onto the image $\plproj(\gamma(\I))$ of $\gamma$. Combining this with the fact that $\minpts_\gamma$ is coarsely Lipschitz (\Cref{lem:coarse_lip_projection}) now easily implies our main result of this section:

\begin{proof}[Proof of \Cref{prop:strong_contract_makes_progress}]
We write $G_t = \gamma(t)$ for $t\in \I$, and fix $s,t\in \I$ with $s\le t$. Since the projection $\fproj\colon \os\to \fc$ is coarsely $80$--Lipschitz \cite[Lemma 2.9]{DT1} and $\gamma$ is a geodesic, we immediately have $d_\fc(G_s,G_t)\le 80\abs{s-t} + 80$ for all $s,t\in \I$. Thus it remains to bound $d_\fc(G_s,G_t)$ from below. Let $\alpha\in \pl^0$ be any primitive conjugacy class represented by an embedded loop in $G_s$ (i.e., any class for which $\alpha\vert G_s \to G_s$ is an embedding). Then $\alpha\in \fproj(G_s)$ by definition of the projection $\fproj\colon \os\to\fc$. Similarly choose $\beta\in \pl^0$ represented by an embedded loop in $G_t$, so that $\beta\in \fproj(G_t)$.

Notice that $\ell(\alpha\vert G_s)\le 1$ and $\ell(\beta\vert G_t) \le 1$ (since the loops are embedded). Thus \Cref{lem:thick strong contract} gives a constant $D_{\epsilon}$, depending only on $\epsilon$ and $D$, such that 
\begin{equation*}
\left.\begin{array}{c} \mathrm{diam}_{\X}(\{G_s\} \cup \minpts_\gamma(\alpha)) \\  \mathrm{diam}_{\X}(\{G_t\}\cup \minpts_\gamma(\beta)) 
\end{array}\right\}\le D_\epsilon
\end{equation*}
Then by \Cref{lem:coarse_lip_projection} we have
\begin{align*}
|s-t| &= d_{\X}(G_s,G_t) \\
&\le \diam_{\X}\big(\{G_s\} \cup  \minpts_\gamma(\alpha)\big) + \diam_{\X}\big(\minpts_\gamma(\alpha) \cup \minpts_\gamma(\beta)\big) + \diam_{\X}\big(\minpts_\gamma(\beta) \cup \{G_t\}\big) \\
&\le 2D_{\epsilon} + D \cdot d_{\pl}(\alpha,\beta) +D \le 2D_{\epsilon} + 2D \cdot d_{\fc}(\alpha,\beta) +D \\
&\le 2D \cdot d_{\fc}(G_s,G_t) + 2D_\epsilon +D.
\end{align*}
This completes the proof.
\end{proof}

\section{Backing into thickness}
\label{sec:backing_up}

In light of \Cref{prop:strong_contract_makes_progress}, to prove our main result \Cref{thm:contraction_impies_progress} it now suffices to show that every nondegenerate strongly contracting geodesic $\gamma$ lives in some definite thick part of $\os$.
We begin by showing that the portion of $\gamma$ where the lengths of primitive loops are minimized is contained in some definite thick part of $\X$. Arguments in \S\S\ref{sec:nondegeneracy}--\ref{sec:finish_up} will then show that all of $\gamma$ must be thick. 

First, recall that $\lipconst$ denotes the coarse Lipschitz constant of the projection $\plproj\colon \X \to \pl$. In particular, $d_\pl (\alpha,\beta) \le \lipconst$ for any $\alpha,\beta \in \pl$ with $\ell(\alpha|G), \ell(\beta|G) \le 2$ for some $G \in \X$.

\begin{proposition}\label{prop:bacl_up}
Let $\gamma\colon \I\to \os$ be a $D$--strongly contracting geodesic and suppose there exist $\alpha_0 \in \pl$ and $s_0 \in \I$ such that $s_0 \in \hat{\rho}_\gamma(\alpha_0)$ and $\len(\alpha_0\vert\gamma(s_0)) \le 2$. If $\diam_\os(\gamma(t_0),\gamma(s_0)) \ge 8D\lipconst$ for some $t_0 < s_0$, then 
\[\gamma(\I\cap (-\infty,s_0])\subset\os_{\epsilon_0}\]
 for some thickness constant $\epsilon_0 > 0$ depending only on $D$.
\end{proposition}
\begin{proof}
Suppose that we are given $s_k\in\I$ and $\alpha_k\in \pl^0$ such that $s_k\in \mintime_\gamma(\alpha_k)$ and $\len(\alpha_k\vert\gamma(s_k))\le 2$. Suppose additionally there exists $t_k < s_k$ with $\diam_\os(\gamma(t_k),\gamma(s_k)) \ge 8D\lipconst$ (note that this holds for $k=0$). We claim there exists an earlier time $s_{k+1} < s_k$ and a conjugacy class $\alpha_{k+1}\in \pl^0$ again satisfying the conditions $s_{k+1}\in \mintime_\gamma(\alpha_{k+1})$ and $\len(\alpha_{k+1}\vert\gamma(s_{k+1})) \le 2$ together with the inequalities
\begin{align}\label{eqn:inductive_ineqs}
\begin{array}{rrcll}
4D\lipconst & \le &  \diam_\os(\gamma(s_{k+1}),\gamma(s_k)) & \le & D(\lipconst + 8D\lipconst^2 + 1), \quad\text{and}\\
3\lipconst  & \le & d_\pl(\alpha_{k+1},\alpha_k)            & \le & \lipconst + 8D\lipconst^2.\end{array}
\end{align}
Indeed, by continuity of $\gamma$ there exists $s_{k+1}' < s_k$ with
\[\diam_\os(x'_{k+1}, x_k) =  8D\lipconst,\]
where $x_k = \gamma(s_k)$ and $x'_{k+1} = \gamma(s'_{k+1})$. Since $\gamma$ is a directed geodesic, there exists a candidate $\alpha_{k+1}$ on $x'_{k+1}$ such that $\len(\alpha_{k+1}\vert \gamma(s'_{k+1} + t)) = e^t\len(\alpha_{k+1}\vert x'_{k+1})$ for all $t \ge 0$. Therefore, if we choose any time $s_{k+1}\in \mintime_\gamma(\alpha_{k+1})$ realizing the minimal length $\minlen_{\alpha_{k+1}}$, we may be assured that $s_{k+1}$ occurs to the left of $s'_{k+1}$. Letting $x_{k+1} = \gamma(s_{k+1})$ we thus find that $s_{k+1} \le s'_{k+1} < s_k$ and $\len(\alpha_{k+1}\vert x_{k+1}) \le \len(\alpha_{k+1}\vert x'_{k+1})\le 2$, as desired.

To prove the claim it remains to verify the inequalities in (\ref{eqn:inductive_ineqs}). First note that $\diam_\os(x_{k+1},x_k)\ge \tfrac{1}{2}\diam_\os(x'_{k+1},x_k) \ge 4D\lipconst$ by the triangle inequality and the fact that $\gamma$ is a directed geodesic. 
Hence, using \Cref{lem:coarse_lip_projection} we see that 
\begin{align*}
d_{\pl}(\alpha_{k+1},\alpha_k) &\ge \frac{1}{D}\big( \diam_\X(x_{k+1},x_k) -D\big)  \ge 3\lipconst.
\end{align*}
On the other hand, we may use the fact that $\ell(\alpha_k\vert x_k), \ell(\alpha_{k+1}|x_{k+1}' )\le 2$ to conclude that 
\begin{align*}
d_\pl(\alpha_{k+1},\alpha_k) \le d_\pl(x'_{k+1},x_k) \le \lipconst\,d_\os(x'_{k+1},x_k) + \lipconst \le 8D\lipconst^2 + \lipconst.
\end{align*}
Another application of \Cref{lem:coarse_lip_projection} then yields,
\begin{align*}
\diam_\X(x_{k+1}, x_k) &\le D \cdot d_{\pl}(\alpha_{k+1},\alpha_k) +D \\
&\le D(\lipconst+8D\lipconst^2)+D,
\end{align*}
which completes the proof of the claim.

Now let $E \colonequals 2D(\lipconst+8D\lipconst^2+1)$ and note that $E\ge 8D\lipconst$ by \eqref{eqn:inductive_ineqs}. Set
\[\epsilon_1 = e^{-E}\qquad\text{and}\qquad\epsilon_0 = e^{-2E}.\]
We claim that $\gamma([s_{k+1},s_k])\subset \os_{\epsilon_0}$. First observe that if $x_k\notin \os_{\epsilon_1}$, then we may find $\beta\in\pl^0$ with $\len(\beta\vert x_k) <\epsilon_1$. In this case (\ref{eqn:inductive_ineqs}) would give $\len(\beta\vert x_{k+1})< 1$ showing that $\beta$ is contained in both projections $\plproj(x_k)$ and $\plproj(x_{k+1})$. However, by (\ref{eqn:inductive_ineqs}), this contradicts the fact that these diameter $\lipconst$ sets contain $\alpha_k$ and $\alpha_{k+1}$, respectively. Whence $x_k\in \os_{\epsilon_1}$ and similarly $x_{k+1}\in\os_{\epsilon_1}$. Another application of (\ref{eqn:inductive_ineqs}) then shows 
\[\len(\beta \vert\gamma(t)) \ge \len(\beta\vert x_k) e^{-\abs{s_k-t}} \ge \epsilon_1 e^{-E} \ge \epsilon_0\]
for all $t\in [s_{k+1},s_k]$. Thus $\gamma([s_{k+1},s_k])\subset \os_{\epsilon_0}$ as claimed.

Let us now prove the proposition. If $\Imin = -\infty$, the above shows that we may find an infinite sequence of times $s_0 > s_1 > \dotsb$ tending to $-\infty$ such that $\gamma([s_{i+1},s_i])\subset \os_{\epsilon_0}$ for each $i>0$. Thus the proposition holds in this case. Otherwise $\Imin \ne -\infty$ and we may recursively construct a sequence $s_0 > \dotsb > s_k$ terminating at a time $s_k\in \I$ for which $\diam_\os(\gamma(\Imin),\gamma(s_k)) <8D\lipconst$ and $\gamma(s_k)\in\os_{\epsilon_1}$ (since $k\ge 1$ by the hypotheses of the proposition). But this implies $\len(\beta\vert\gamma(t))\ge \epsilon_1 e^{-8D\lipconst}$ for every conjugacy class $\beta$ and time $t\in [\Imin,s_k]$. Thus $\gamma([\Imin,s_k])\subset\os_{\epsilon_0}$ as well and the proposition holds.
\end{proof}
 
\section{Nondegeneracy and thickness}
\label{sec:nondegeneracy}

We have now developed enough tools to both establish nondegeneracy for typical strongly contracting geodesics and to show that each nondegenerate strongly contracting geodesic has a uniformly thick initial segment. We first establish \Cref{lem:usually_nondegenerate}, which implies that strongly contracting geodesics are automatically nondegenerate except possibly in the case of a short geodesic with a very thin left endpoint:

\begin{lemma}
\label{lem:usually_nondegenerate}
Suppose that $\gamma\colon\I\to \os$ is a $D$--strongly contracting geodesic. Then either of the following conditions imply that $\gamma$ is nondegenerate:
\begin{itemize}
\item $\abs{\I} \ge A$ for some constant $A$ depending only on $D$ and the injectivity radius of $\gamma(\Imin)$.
\item $\I$ is an infinite length interval
\end{itemize}
\end{lemma}

\begin{proof}
Let $\gamma\colon\I\to\os$ be a $D$--strongly contracting geodesic. First suppose that $\gamma$ is not infinite to the left (i.e., that $\Imin\neq -\infty$) and let $\epsilon$ be the injectivity radius of $\gamma(\Imin)$. Take $A = \sym_{\epsilon'}18D\lipconst$, where $\epsilon' = \epsilon e^{-18D\lipconst}$. We claim that $\gamma$ is nondegenerate provided $\abs{\I}\ge A$; this will establish the first item of the lemma. Indeed, consider the points $H = \gamma(\Imin)$ and $G = \gamma(\Imin+A)$. If $G\notin \os_{\epsilon'}$, then we automatically have $d_\os(G,H)\ge \log(\epsilon/\epsilon')\ge 18D\lipconst$ by definition of the Lipschitz metric, and otherwise $G\in\os_{\epsilon'}$ so that $d_\os(G,H)\ge d_\os(H,G)/\sym_{\epsilon'} = 18D\lipconst$ by \Cref{lem: symmetric_in_thick}.

To prove the second item of the lemma, it remains to consider the case $\Imin = -\infty$. Choose $s_0\in \I$ arbitrarily and let $\alpha\in \pl^0$ be a primitive conjugacy class with $\len(\alpha\vert \gamma(s_0))\le 2$ (e.g., a candidate). Next choose a time $s\in \mintime_\gamma(\alpha)$ and note that $\len(\alpha\vert\gamma(s))\le 2$. The fact that $\Imin=-\infty$ ensures we may find $t< s$ such that $\diam_\os(H, G)\ge  \sym_{\epsilon_0}18D\lipconst$, where $H = \gamma(t)$, $G = \gamma(s)$ and  $\epsilon_0>0$ is the thickness constant from \Cref{prop:bacl_up}. Then $\gamma(\I\cap(-\infty,s])$ lies in $\os_{\epsilon_0}$ by \Cref{prop:bacl_up}, and so we may conclude $d_\os(G,H)\ge \diam_\os(H,G) / \sym_{\epsilon_0} \ge 18D\lipconst$ by \Cref{lem: symmetric_in_thick}.
\end{proof}

Our next task is to show that nondegeneracy implies that the hypotheses of \Cref{prop:bacl_up} are satisfied, and consequently that the initial portion of any such geodesic is uniformly thick. The following lemma will aid in this endeavor.

\begin{lemma} \label{prop:local}
Let $\gamma: \I \to \X$ be a $D$--strongly contracting geodesic in $\X$ and suppose that there are $\alpha \in \pl^0$ and $s,t_1 \in \I$ such that $s \le t_1$ and $\ell(\alpha|\gamma(t_1)) < e^{-D}\ell(\alpha|\gamma(s))$. Then $\alpha$ has its length minimized to the right of $s \in \I$, i.e. $s < r$ for all $r\in \mintime_\gamma (\alpha)$.
\end{lemma}
\begin{proof}
Set $H = \gamma(t_1)$ and $J = \I \cap [s-D,s]$. 
Then $\alpha$ can stretch by at most $e^D$ along $J$ (since $\gamma$ is a directed geodesic), and so for each $j\in J$ we have
\[\len(\alpha\vert\gamma(j)) \ge e^{-D}\len(\alpha\vert \gamma(s)) > \len(\alpha\vert H).\]

Let $r\in \mintime_\gamma(\alpha)$ be any time minimizing the length of $\alpha$. Fix a marked rose $R$ with a petal corresponding to the conjugacy class $\alpha$. For $0 < \sigma < 1$, let $R_\sigma$ denote the metric graph obtained from $R$ by setting the length of the $\alpha$--petal to $\sigma$ and the length of each other petal to $\frac{1-\sigma}{r-1}$. As in the proof of \Cref{lem:coarse_lip_projection}, $\sigma$ can be taken sufficiently small so that $\alpha$ is the candidate of $R_\sigma$ realizing the distance from $R_\sigma$ to $H$. Consequently, if $[R_\sigma,H]$ denotes a directed geodesic from $R_\sigma$ to $H$, then $\alpha$ also realizes the distance from $G$ to $H$ for each point $G\in [R_\sigma,H]$. It now follows that for each $j\in J$ and $G\in [R_\sigma,H]$ we have
\[d_\os(G,\gamma(j)) \ge \log\left(\frac{\len(\alpha\vert \gamma(j))}{\len(\alpha\vert G)}\right) > \log\left(\frac{\len(\alpha\vert H)}{\len(\alpha\vert G)}\right) = d_\os(G,H).\]
In particular the entire projection $\cp_\gamma([R_\sigma,H])$ is disjoint from the interval $\gamma(J)$. 

Taking $\sigma$ smaller if necessary, we may also assume that $\alpha$ realizes the distance from $R_\sigma$ to $\gamma(r)$. Since $\len(\alpha\vert\gamma(r)) = \minlen_\alpha$ is the minimal length of $\alpha$, this forces $\gamma(r)\in\cp_\gamma(R_\sigma)$. Whence $r$ cannot lie in $J$ by the above. Now, if $J = \I\cap[s-D,s]$ contains the initial endpoint $\Imin$, this observation forces $r > s$ as desired. Otherwise $J$ is the length--$D$ interval $J= [s-D,s]$, and we may apply \Cref{lem:dense_projection} to conclude that $r\in \cp_\gamma([R_\sigma,H])$ is contained in $\gamma(\I\cap (s,\infty))$. Thus $r > s$ and the lemma holds.
\end{proof}

\begin{corollary} \label{cor:nondegen_implies_min}
Given $D$ there exists $\epsilon_0> 0$ with the following property. If $\gamma\colon\I\to\os$ a nondegenerate $D$--strongly contracting geodesic, then there exists $s_0\in \I$ such that $\gamma(\I\cap (-\infty,s_0]) \subset \X_{\epsilon_0}$.
\end{corollary}
\begin{proof}
By definition, nondegeneracy implies that there are times $t_0 <t_1$ in $\I$ so that $d_\X(\gamma(t_1),\gamma(t_0)) \ge 18D\lipconst$. Letting $s \in \I$ be such that $t_0 \le s \le t_1$ and $d_\X(\gamma(t_1), \gamma(s)) = 2D$, the triangle inequality then gives
\[d_\os(\gamma(s),\gamma(t_0)) \ge d_\os(\gamma(t_1),\gamma(t_0)) - 2D \ge 16D\lipconst.\]
Let $\alpha \in \pl^0$ denote the candidate of $\gamma(t_1)$ that realizes the distance to $\gamma(s)$, i.e. $\ell(\alpha|\gamma(t_1)) = e^{-2D}\ell(\alpha|\gamma(s))$. If we choose any time $s_0\in\mintime_\gamma(\alpha)$ minimizing $\len(\alpha| \gamma(\cdot))$, then $s \le s_0$ by \Cref{prop:local}. Since $\ell(\alpha| \gamma(t_1)) \le 2$, we have that $\ell(\alpha | \gamma(s_0)) \le 2$. Finally, since $t_0 < s\le s_0$ and $\gamma$ is a directed geodesic, we find that
\begin{align*}
16D\lipconst \le d_\os(\gamma(s),\gamma(t_0)) &\le d_\os(\gamma(s),\gamma(s_0)) + d_\os(\gamma(s_0),\gamma(t_0))\\
&\le d_\os(\gamma(t_0),\gamma(s_0)) + d_\os(\gamma(s_0),\gamma(t_0)) \le 2\diam_\os(\gamma(t_0),\gamma(s_0)).
\end{align*}
Therefore $\diam_\os(\gamma(t_0),\gamma(s_0))\ge 8D\lipconst$ and we may apply \Cref{prop:bacl_up} to complete the proof.
\end{proof}

Finally, we show that if a strongly contracting geodesic in $\X$ has its initial portion contained in some definite thick part of $\X$, then the entire geodesic remains uniformly thick.

\begin{lemma}
\label{prop:left_thick}
Suppose that $\epsilon_0, D >0$ and that $\gamma: \I \to \X$ is a $D$--strongly contracting geodesic with $\gamma(\I \cap (-\infty, b]) \subset \X_{\epsilon_0}$ for some $b\in \I$. Then $\gamma(\I) \subset \X_{\epsilon}$, for 
$\epsilon = \tfrac{\epsilon_0}{2}e^{-4D\lipconst}$.
\end{lemma}
\begin{proof}
Write $G_t = \gamma(t)$ for $t\in \I$. Without loss of generality we assume $\epsilon_0 < 1$. 
It suffices to prove $\minlen_\alpha \ge \epsilon$ where $\alpha$ is an arbitrary primitive loop $\alpha$. Note that $\minlen_\alpha >0$ and $\minpts_\gamma(\alpha)\neq\emptyset$ by  \Cref{lem:coarse_lip_projection}. If $\minlen_\alpha \ge \epsilon_0/2$ then we are done. Otherwise we choose $t_\alpha\in\mintime_\gamma(\alpha)$ and note that $\len(\alpha\vert G_{t_\alpha})< \epsilon_0/2$. 
Since $\ell(\alpha\vert G_t)$ is continuous in $t$ and at least $\epsilon_0$ for all $t \le b$, there is some $s < t_\alpha$ so that $\len(\alpha\vert G_s) = \epsilon_0$.

Let $\beta$ be a candidate of $G_s$ such that $\len(\beta \vert G_{s+t}) = e^t\len(\beta \vert G_s)$ for all $t > 0$. If $r\in \mintime_\gamma(\beta)$ is any time minimizing the length of $\beta$, we then necessarily have $r \le s$. Since $\alpha$ and $\beta$ each have length less than $2$ at $G_s\in\os$, it follows that $d_\pl (\alpha,\beta) \le \lipconst$. \Cref{lem:coarse_lip_projection} then implies that 
\[\diam_\os (G_r,G_{t_\alpha}) \le D \cdot d_\pl(\alpha,\beta) + D \le D\lipconst + D \le 2D\lipconst.\] 
Since $\gamma$ is a directed geodesic, $d_\X(G_r,G_s) \le d_\X(G_r,G_{t_\alpha})$ and so
\begin{align*}
d_\X (G_{t_\alpha},G_s) &\le d_\X(G_{t_\alpha},G_r) + d_\X(G_r,G_s) \\
&\le d_\X(G_{t_\alpha},G_r) + d_\X(G_r,G_{t_\alpha}) \le 4D\lipconst.
\end{align*}
In particular, $\frac{\ell(\alpha|G_{s})}{\ell(\alpha|G_{t_\alpha})} \le e^{4D\lipconst}$, and so we find $m_{\alpha} \ge \epsilon_0 e^{-4D\lipconst}$ as desired.
\end{proof}

\section{Characterizing strongly contracting geodesics}
\label{sec:finish_up}

We now combine the previous results to complete the proof of our main theorem:

\begin{thm:main}[Strongly contracting geodesics make progress in $\fc$]
For each $D> 0$ there exist constants $K\ge 1$ and $\epsilon> 0$ with the following property. If $\gamma\colon \I\to \os$ is a nondegenerate $D$--strongly contracting geodesic, then $\gamma(\I)$ lies in the $\epsilon$--thick part $\os_\epsilon$ and $\fproj\circ\gamma\colon \I\to \fc$ is a $K$--quasigeodesic.
\end{thm:main}
\begin{proof}
Suppose $\gamma\colon\I\to\os$ is nondegenerate and $D$--strongly contracting. By nondegeneracy, \Cref{cor:nondegen_implies_min} and \Cref{prop:left_thick} together give $\gamma(\I)\subset\os_\epsilon$ for some $\epsilon = \epsilon(D) > 0$. \Cref{prop:strong_contract_makes_progress} then provides a constant $K = K(D,\epsilon)$ for which $\fproj\circ\gamma\colon\I\to\fc$ is a $K$--quasigeodesic.
\end{proof}

We next discuss the converse \Cref{th: BF_contract} and explain how the Bestvina--Feighn result \cite[Corollary 7.3]{BFhyp} on folding lines that make definite progress in $\fc$ may be promoted to arbitrary geodesics. While this promotion essentially follows from our earlier work \cite{DT1}, we have opted to include a proof here for completeness. In this discussion we assume the reader is familiar with folding paths and standard geodesics in $\X$; for background on this see \cite{FMout, BFhyp, DT1}.

\begin{thm:progress}[Progressing geodesics are strongly contracting] 
Let $\gamma\colon \I \to \X$ be a geodesic whose projection to $\F$ is a $K$--quasigeodesic. Then there exists $D > 0$ depending only on $K$ (and the injectivity radius of the terminal endpoint of $\gamma$) such that $\gamma$ is $D$--strongly contracting in $\X$.
\end{thm:progress}
\begin{proof}
Let $\gamma\colon\I\to\os$ be an arbitrary directed geodesic such that $\fproj\circ\gamma\colon\I\to\fc$ is a $K$--quasigeodesic, and let $H,H'\in\os$ be metric graphs satisfying $d_\os(H,H')\le d_\os(H,\gamma)$. Lemma 4.3 of \cite{DT1} shows that $\gamma(\I)\subset \os_{\epsilon}$ for some $\epsilon > 0$ depending only on $K$ (and the injectivity radius of $\gamma(\Ipl)$ when $\Ipl \ne +\infty$). Using the coarse symmetry of $d_\os$ in $\os_{\epsilon}$ (\Cref{lem: symmetric_in_thick}), one may easily show that $\cp_\gamma(G)$ is never empty. Hence to prove the theorem it suffices to choose $p\in \cp_\gamma(H)$ and $p'\in\cp_\gamma(H')$ arbitrarily and bound $\diam_\os(p,p')$ in terms of $K$ and $\epsilon$.

Choose a finite subinterval $\J=[a,b]\subset\I$ with $p,p'\in\gamma(\J)$ and let $\bar{\gamma} = \gamma\vert_{\J}$. Notice that $p\in\cp_{\bar\gamma}(H)$ and $p'\in \cp_{\bar\gamma}(H')$. If $\rho\colon\J\to\os$ is any standard geodesic from $\bar\gamma(a)$ to $\bar\gamma(b)$, 
then Theorem 4.1 of \cite{DT1} ensures $\fproj\circ\rho$ is a $K'=K'(K,\epsilon)$--quasigeodesic and that $\rho(\J)\subset \os_{\epsilon'}$, where  $\epsilon'=\epsilon'(K,\epsilon)$. Consequently, Proposition 7.2 of \cite{BFhyp} and Lemma 4.11 of \cite{DT1} (see also \cite[Proposition 2.11]{DT1} and the following remark) immediately show that $\rho$ is $D'=D'(K,\epsilon)$--strongly contracting. 

Theorem 4.1 of \cite{DT1} additionally shows that $\bar\gamma(\J)$ and $\rho(\J)$ have symmetric Hausdorff distance at most $A' = A'(K,\epsilon)$. Consequently, we claim that there exists $B' = B'(A',D',\epsilon')$ such that 
\begin{equation}
\label{eqn:projections_near_each_other}
\diam_\os(\cp_{\bar\gamma}(G)\cup\cp_\rho(G))\le B'
\end{equation}
for all $G\in \os$. To see this, choose $Y_0\in\cp_{\bar\gamma}(G)$ arbitrarily and let $Y\in \cp_\rho(Y_0)$ be a closest point projection of $Y_0$ to $\rho$. Noting that $d_\os(G,Y_0)\le d_\os(G,\rho) + A'$ and $d_\os(Y_0,Y)\le A'$, we see that $d_\os(G,Y)\le d_\os(G,\rho) + 2A'$. Thus, as in the proof of \Cref{lem:thick strong contract}, we may find $Y'\in \os$ along a directed geodesic from $G$ to $Y$ such that $d_\os(G,Y')\le d_\os(G,\rho)$ and $d_\os(Y',Y)\le 2A'$. The strong contraction property for $\rho$ now gives $\diam_\os(\cp_\rho(Y')\cup\cp_\rho(G))\le D'$, and the fact that $Y$ is $\epsilon'$--thick and near $Y'$ bounds $\diam_\os(\{Y\}\cup\cp_\rho(Y'))$ in terms of $\epsilon'$ and $A'$. Hence $\diam_\X(\{Y\}\cup\cp_\rho(G))$ is bounded and, since $\diam_\os(Y,Y_0)\le A'$, the claimed inequality (\ref{eqn:projections_near_each_other}) holds.

We next claim that $\cp_\rho$ is coarsely $1$--Lipschitz. That is, there exists $C'=C'(D',\epsilon')$ such that
\begin{equation}
\label{eqn:coarse_lipschitz_proj}
\diam_\os(\cp_\rho(G_1)\cup\cp_\rho(G_2)) \le \diam_\os(G_1,G_2) + C'
\end{equation}
for all $G_1,G_2\in\os$. Indeed, first consider the case that there exists a directed geodesic $[G_1,G_2]$ with $d_\os(Y,\rho) \ge D'$ for all $Y\in [G_1,G_2]$. Dividing $[G_1,G_2]$ into $n=\ceil{d_\os(G_1,G_2)/D'}$ subgeodesics of equal length (at most $D'$) and applying strong contraction to each, one finds that
\[\diam_\os(\cp_\rho(G_1)\cup\cp_\rho(G_2)) \le n D' \le \left(\frac{d_\os(G_1,G_2)}{D'} + 1\right) D' \le  \diam_\os(G_1,G_2)+D'.\]
Next consider the case that $d_\os(G_i,\rho)\le D'$ for each $i=1,2$. Choosing $G_i'\in\cp_\rho(G_i)$ arbitrarily, \Cref{lem: symmetric_in_thick} and the thickness of $G'_i\in\os_{\epsilon'}$ together bound $\diam_\os(G_i,G'_i)$ in terms of $\epsilon'$ and $D'$. Thus the difference $\abs{\diam_\os(G'_1,G'_2) - \diam_\os(G_1,G_2)}$ is bounded in terms of $\epsilon'$ and $D'$. The general case now follows by subdividing an arbitrary directed geodesic $[G_1,G_2]$ into at most three subgeodesics that each fall under the cases considered above.

To complete the proof of the theorem, note that since $d_\os(H,\rho)\ge d_\os(H,\bar{\gamma})-A'$ we can find a point $H_0\in \os$ (say on a geodesic from $H$ to $H'$) such that $d_\os(H,H_0)\le d_\os(H,\rho)$ and $d_\os(H_0,H')\le A'$. Then $\diam_\os(\cp_\rho(H)\cup\cp_\rho(H_0))\le D'$ by strong contraction and $\diam_\os(\cp_\rho(H_0)\cup\cp_\rho(H')) \le A'+C'$ by (\ref{eqn:coarse_lipschitz_proj}). Combining these with (\ref{eqn:projections_near_each_other}) gives the desired bound on $\diam_\os(p,p')$.
\end{proof}

\section{Contracting subgroups of $\Out(\free)$}
\label{sec:contracting_subgroups}

In this section we apply \Crefrange{th: BF_contract}{thm:contraction_impies_progress} to characterize the finitely generated subgroups of $\Out(\free)$ that quasi-isometrically embed into $\fc$. Recall that a subgroup $\Gamma\le\Out(\free)$ is said to be \define{contracting in $\os$} if there exists $R\in\os$ and $D > 0$ such that any two points in the orbit $\Gamma\cdot R$ are joined by a $D$--strongly contracting geodesic. Using \Cref{th: BF_contract,thm:contraction_impies_progress} and \cite[Theorem 4.1]{DT1}, one may show that this definition is in fact equivalent to the following stronger condition: for each $R\in\os$ there exists $D > 0$ such that every directed geodesic between points of $\Gamma\cdot R$ is $D$--strongly contracting; alternately, under the hypothesis of \Cref{th:contracting_orbit}, this follows from the proof below.

\begin{thm:orbits}[Contracting orbits]
Suppose that $\Gamma \le \Out(\free)$ is finitely generated and that the orbit map $\Gamma \to \X$ is a quasi-isometric embedding. Then $\Gamma$ is contracting in $\os$ if and only if the orbit map $\Gamma\to\fc$ to the free factor complex is a quasi-isometric embedding.
\end{thm:orbits}
\begin{proof}
The ``if'' direction was essentially obtained by the authors in \cite{DT1}:  Supposing that $\Gamma$ admits an orbit map into $\F$ that is a quasi-isometric embedding, Theorem $5.4$ of \cite{DT1} implies that for each $R\in \os$ the orbit $\Gamma \cdot R$ is $A$--quasiconvex for some $A>0$. This means that any directed geodesic $\gamma\colon\I\to\os$ between orbit points lies in the symmetric $A$--neighborhood of $\Gamma\cdot R$. Since $\Gamma\to\fc$ is a quasi-isometric embedding, it follows easily that the projection $\fproj\circ\gamma\colon\I\to\fc$ is a parameterized quasigeodesic with uniform constants. Therefore $\gamma$ is uniformly strongly contracting by \Cref{th: BF_contract}.

For the ``only if'' direction, suppose that $\Gamma$ is contracting with respect to $R\in \os$ and $D > 0$ and that the assignment $g\mapsto g\cdot R$ defines a $C$--quasi-isometric embedding. Choose $g,h\in \Gamma$ and let $\gamma\colon[a,b]\to \os$ be a $D$--strongly contracting geodesic from $g\cdot R$ to $h\cdot R$. \Cref{prop:left_thick} then ensures $\gamma([a,b])\subset \os_\epsilon$ for some $\epsilon>0$ depending on $D$ and the injectivity radius of $R$, and so \Cref{prop:strong_contract_makes_progress} implies that $\fproj\circ\gamma$ is a $K=K(D,\epsilon)$--quasigeodesic. Since $d_\Gamma(g,h)$ and $d_\fc(g\fproj(R),h\fproj(R))$ both coarsely agree with $d_\os(\gamma(a),\gamma(b)) = d_\os(g\cdot R,h\cdot R)$, there is a constant $E = E(K,C)\ge 1$ such that
\[\frac{1}{E} d_\Gamma(g,h) - E \le \diam_\fc(g\fproj(R),h\fproj(R)) \le E\,d_\Gamma(g,h) + E\]
Thus the assignment $g\mapsto g\cdot A$, where $A\in\fproj(R)$, defines a quasi-isometric embedding $\Gamma\to \fc$.
\end{proof}

\bibliographystyle{alphanum}
\bibliography{strongly_contracting_in_OS}

\bigskip

\noindent
\begin{minipage}{.55\linewidth}
Department of Mathematics\\
Vanderbilt University\\
1326 Stevenson Center\\
Nashville, TN 37240, U.S.A\\
E-mail: {\tt spencer.dowdall@vanderbilt.edu}
\end{minipage}
\begin{minipage}{.45\linewidth}
Department of Mathematics\\ 
Temple University\\ 
1805 North Broad Street\\ 
Philadelphia, PA 19122, U.S.A\\
E-mail: {\tt samuel.taylor@temple.edu}
\end{minipage}

\end{document}